\documentclass{article}
\usepackage{dsfont}

\usepackage{verbatim}

\usepackage{times}
\usepackage[pdftex]{graphicx}
\usepackage{amssymb,amsfonts,amsmath,amsthm}
\usepackage{float}
\usepackage{xcolor}
\usepackage{enumitem}
\usepackage{bbm}
\usepackage{url}

\usepackage{caption}
\usepackage{subcaption}

\usepackage{titlesec}

\titleformat{\paragraph}
{\normalfont\normalsize\bfseries}{\theparagraph}{1em}{}
\titlespacing*{\paragraph}
{0pt}{3.25ex plus 1ex minus .2ex}{1.5ex plus .2ex}

 \DeclareMathOperator*{\esssup}{ess\,sup}
 
\DeclareMathOperator*{\argmax}{arg\,max}
\DeclareMathOperator*{\argmin}{arg\,min}

\newtheorem{theorem}{Theorem}
\newtheorem{proposition}{Proposition}
\newtheorem{lemma}{Lemma}
\newtheorem{ass}{Assumption}
\newtheorem{definition}{Definition}
\newtheorem{remark}{Remark}

\newcommand\EE {\mathbb E}
\newcommand\FF {\mathbb F}
\newcommand\NN {\mathbb N}
\newcommand\RR {\mathbb R}
\newcommand\PP {\mathbb P}

\newcommand\WW {\mathbb W}

\def\qed{\hskip6pt\vrule height6pt width5pt depth1pt}

\title{Sequential optimal contracting in continuous time}
\author{Guillermo Alonso Alvarez\thanks{Department of Mathematics, University of Michigan. \protect\url{guialv@umich.edu}} \and Erhan Bayraktar\thanks{Department of Mathematics, University of Michigan. \protect\url{erhan@umich.edu}}\thanks{Funded in part by the NSF through DMS-2106556 and by the Susan M. Smith Chair.} \and Ibrahim Ekren\thanks{Department of Mathematics, University of Michigan. \protect\url{iekren@umich.edu}}\thanks{Funded in part by the NSF through DMS-2406240.} \and  Liwei Huang\thanks{Department of Mathematics, University of Michigan. \protect\url{huanglw@umich.edu}}}
\date{\today}
\begin{document}
\maketitle
\begin{abstract}
In this paper we study a principal-agent problem in continuous time with multiple lump-sum payments (contracts) paid
at different deterministic times. We reduce the non-zero sum Stackelberg game between the principal and agent to a standard stochastic optimal control problem. We apply our result to a benchmark model for which we investigate how different inputs (payment frequencies, payments' distribution, discounting factors, agent's reservation utility) affect the principal's value and agent's optimal compensations.
\end{abstract}
\section{Introduction}

 Principal-agent problems describe the strategic interactions between two parts: the principal (i.e. the manager), and the agent (i.e. the employee). The principal aims to incentivize good performance from the agent by designing an optimal compensation, i.e. the contract. Such contracts are contingent on an agent's controlled random state, referred as the output process (i.e. the firm's value), and must satisfy the agent's participation constraint. In most of real-life applications, the contracts are structured as multiple lump-sum payments scheduled periodically (e.g. insurance, brokerage, managerial compensations). The contract schedule can be deterministic (e.g. salaries) or random (e.g. spot bonuses). This paper investigates the problem of optimal contract schedule design (multiple lump-sum payments) in continuous time. To our knowledge so far, this is the first paper addressing this problem as prior studies focused on either (i) continuous time lump-sum single payment and a possible continuous payment, (ii) discrete time multiple lump-sum payments. A continuous time approach facilitates the tractability and the qualitative analysis in the study of optimal contract schedule design. The later has numerous applications in economics, finance and management science...     
 
A substantial amount of research has been developed in the continuous time principal-agent literature. The first publication on this matter is authored by Holmström and Milgrom \cite{holmstrom1987aggregation}. The authors demonstrated that in a finite horizon setting the optimal contract is linear with respect to the output process.  Since then, multiple extensions have been developed. We highlight the following papers extending the original model: 
\cite{schattler1993first}, \cite{sung1995linearity},\cite{sung1997corporate}, \cite{muller1997first}, \cite{muller2000asymptotic}, \cite{hellwig2002discrete}, \cite{hellwig2007role}. 
 Later on, the following articles 
\cite{pham2009continuous}, \cite{williams2011persistent}, \cite{williams2015solvable},  \cite{cvitanic2006optimal}, \cite{cvitanic2008principal}, \cite{cvitanic2009optimal}  employed the stochastic maximum principle and forward–backward stochastic differential equations to characterize optimal compensations in a more general setting. Finally, we mention the latest contributions:   \cite{competitive}, \cite{manymanyagents}, \cite{hernandez2019contract}, \cite{hernandezpossmai}. 

In discrete time, the problem of optimal contracting schedule design has been extensively explored. Most papers, (e.g.  \cite{rubinstein1983repeated}, \cite{rogerson1985repeated}), focus on the infinite repeated contracting scenario, where neither the principal nor the agent discounts the future. Lambert \cite{lambert1983long} addressed the finite multi-period contracts, considering the agent's problem with discounting. Fudenberg \cite{fudenberg1990short}, \cite{fudenberg1990moral} further examined the distinctions between short-term and long-term contracts, identifying the conditions under which the complexity of optimal long-term contracts is no greater than that of short-term contracts. The primary techniques employed in the previous models are the first-order approach and variational techniques.

Our paper has the following contributions: Firstly, in section \ref{setup} and \ref{principal-agent} we consider a general sequential contracting problem with contract schedules consisting of multiple lump-sum payments paid at deterministic times. Following a similar approach than \cite{cvitanic2018dynamic}, we represent the agent's value using a recursive system of quadratic BSDEs (backward stochastic differential equations). The previous allows us to 
reduce the principal's bi-level optimization problem to a single stochastic control problem with mixed static and continuous controls which can be approached using dynamic programming techniques (see, for example, \cite{capponizhang}). The assumptions made are fairly general and can be applied to a variety of sequential contracting problems.  

Secondly, in section \ref{benchmark} we introduce a benchmark model to investigate how the model's exogenous inputs affect the following outputs: principal and agent's value, optimal contract schedules and optimal frequency and distribution of payments. Inspired by \cite{sannikov2008continuous}, our model assumes a risk neutral principal and a risk averse (via power utility), time sensitive agent.  In addition to the reservation utility we assume a limited liability constraint imposing that every admissible contract schedule provides to the agent a non-negative continuation utility. The principal´s problem is to find an optimal contract schedule satisfying the agent's reservation utility and the limited liability constraints.

Following the results from sections \ref{setup} and \ref{principal-agent} we reduce the principal's sequential contracting problem to a stochastic control problem with state constraints and mixed continuous-discrete controls. The previous can be approach by a careful application of the results from Bouchard and Nutz \cite{bouchard2012weak}. In their paper, Bouchard and Nutz derive a dynamic programming principle for the state constrained stochastic control problems and write a comparison principle for viscosity solutions of its associated HJB equation. In order to apply the aforementioned results, we find an appropriate upper and lower bound of the principal´s value function using probabilistic arguments. Finally, in Theorem \ref{hjb:pp:in} we show that the principal´s value function is continuous and solves a recursive system of HJB equations. See \cite{usersguideviscosity} for a comprehensive review of the theory of viscosity solutions of second order partial differential equations.  

In section \ref{numerics} we approximate numerically the solution of the recursive system of HJB equations characterizing the principal's value function. The previous allows us to investigate how the contracting environment affects the principal's value and the agent's optimal compensations. Firstly, we observe from our numerical results that the optimal intermediate payments $\xi_i^{*}$ delivered to the agent are an increasing function of the agent's continuation utility. Furthermore, as the agent's discounting factor increases, we conclude that the principal's maximum achievable profit decreases. This is connected to the concept of employment interval described by Sannikov \cite{sannikov2008continuous} indicating that the principal is less inclined to provide a large utility to an impatient agent. Moreover, independently from the agent's level of patience, the principal always benefits from increasing the frequency of payments. The previous is a consequence of the absence of transaction costs in our model. In what respects to the distribution of the payments, we conclude that the principal's choice of payment distribution highly depends on the agent's participation constraint. The previous effect is noticeable when the agent is impatient. In what concerns to the distribution of payments, we conclude that the principal's optimal payment distribution distribution highly depends on the agent's reservation utility level. The previous effect is noticeable when the agent is impatient.

Finally, we compare the principal´s value in our model (initial negotiation) with an analogous contracting model when the contracts are renegotiated at every transaction time. In our benchmark model the negotiation is finalized at the initial time implying that the agent precommits to a long-term contract guaranteeing to remain in the firm until the last payment $\xi_{N}$ is transacted. In contrast, the renegotiation problem represents the scenario where the principal and agent sign a new contract at beginning of each contracting period with a possibly different agent's reservation utility level. We show from our numerical results that the principal's optimal negotiation setting (initial negotiation, renegotiation) depends on the model's inputs.

\section{ Set up}\label{setup}
We start by describing the dynamics of the output process, i.e the value of the firm. We consider
$\Omega:= {C}([0,T],\RR)$  the space of $\RR$-valued, continuous functions endowed with the supremum norm $\|\cdot\|_{{C}([0,T],\RR)}$, and the Wiener measure $\WW$. We denote by $W$ the canonical random element in $\Omega$, and $\FF$ its augmented filtration (with respect to $\WW$). The
 principal compensates the agent with $N$ lump sum payments $\bar{\xi}_N:=(\xi_1,\ldots,\xi_N)$, i.e. the contract schedule, transacted at $\mathcal{T}_N:= \{T_1,\ldots,T_{N-1}\},$ where $0 < T_1< \ldots < T_{N-1} < T < \infty$. For each $i \in \left\{1,\ldots,N\right\}$, we define the set of $i$-th contracts as follows: 
 \begin{equation}\label{set.contracts}
    \mathcal{C}^0_i := \left\{ \xi : (\Omega,\mathcal{F}_{T_i}) \mapsto \left(E,\mathcal{B}(E)\right)  \text{ measurable }, \text{ } \EE^{\WW}\left[\exp\left(p\xi \right)\right]< \infty, \text{ for all } p\geq 0 \right\}, 
 \end{equation}
 where $E \subset \RR$ is a non empty convex set. 
 
 Additionally, for all $i \in \left\{1,\ldots,N\right\}$, we introduce   $\Sigma_i :=\prod_{j=1}^i \mathcal{C}_j^0 $ corresponding to the set of contract schedules $\bar{\xi}_i :=(\xi_1,\ldots,\xi_i)$ transacted at $\mathcal{T}_i:= \{ T_1,\ldots,T_i\} \subset \mathcal{T}_N$. Let $\bar{\xi}_{N-1} \in \Sigma_{N-1}$ be a contract schedule, the output process $X^{\bar{\xi}_{N-1}} $ is the unique strong solution of the following iterative system of SDEs (stochastic differential equations):
 \begin{equation}\label{state_base}
     X^{\bar{\xi}_{N-1}}_t = x_0+ \sum_{i=1}^{N-1} G_i\left(X^{\bar{\xi}_{N-1}}_{T_i^{-}},\xi_i\right)\mathbbm{1}_{t\geq T_i}+ \int_{0}^t \sigma(s,X^{\bar{\xi}_{N-1}}_s) dW_s, \quad \WW-a.s.
 \end{equation}
  
where $x_0 \in \RR$, and $\sigma : [0,T] \times \RR \mapsto \RR$,   $G_i:\RR \times E \mapsto \RR$, $1\leq i \leq N$, are measurable mappings satisfying the following assumption: 
 \begin{ass}\label{ass1}
 The mappings $(G_i)_{i=1}^N$, and $\sigma$ satisfy the following properties: 
 \begin{itemize}
             \item[(i)] There exists a constant $K_G>0$, for which
    \begin{equation*}
    |G_i(x,y)| \leq K_G \left(1+ |x|+|y| \right),
    \end{equation*}
  for all $(i,x,y) \in \{1,\ldots,N\} \times \RR \times E$. 
  \item[(ii)] There exists a constant $K_\sigma>0$, for which 
      $$|\sigma(t,x)|\leq K_\sigma\left(1+|x|\right), \quad  |\sigma(t,x)-\sigma(t,x')|\leq K_\sigma|x-x'|,$$
      for all $x, x' \in \RR$. 
    \item[(iii)]$\sigma$ is bounded, and $\sigma_t$ is invertible with bounded inverse for all $0 \leq t\leq T$.
  \end{itemize}
 \end{ass}
Note that, under Assumption \ref{ass1}, the SDE \eqref{state_base} has a unique strong solution. We write $X_{t^{-}} := \lim_{s\rightarrow t^{-}} X_s$ for all $0 \leq t \leq T$. Next, we describe the impact of the agent's actions on the output process. We assume that the agent's actions are $\FF$-adapted, $\RR$-valued process $\alpha$. Furthermore, we introduce the measurable process $\lambda: [0,T]\times \RR \times \RR \mapsto \RR$, satisfying the following condition: 
\begin{ass}
\label{ass2} 
$\lambda(t,x,\cdot)\in {C}^1(\RR)$ for all $(t,x) \in [0,T]\times \RR$.
Moreover, for all $(t,a,x) \in [0,T]\times \RR^2$, there exists $K_{\lambda} >0$, such that 
    \begin{align*}
        |\lambda(t,x,a)| \leq K_{\lambda}\left(1+|a|\right),  \quad
        \left|\frac{\partial }{\partial a}\lambda(t,x,a)\right|\leq K_{\lambda }.
    \end{align*}
\end{ass}
Given a contract schedule $\bar{\xi}_{N-1}\in \Sigma_{N-1}$, we introduce the set of admissible agent's actions $\mathcal{A}(\bar{\xi}_{N-1})$, consisting of the $\FF$-adapted processes $\alpha$ for which there exists $\epsilon >0$, such that 
\begin{equation*}
   \EE\left[ \mathcal{E}\left(\int_0^T \sigma_s(X_s^{\bar{\xi}_{N-1}})^{-1}\lambda_s(X^{\bar{\xi}_{N-1}}_s,\alpha_s)dW_s \right)^{1+\epsilon} \right] < \infty. 
\end{equation*}

Using Girsanov theorem, we obtain that for any $(\bar{\xi}_{N-1},\alpha) \in \Sigma_{N-1}\times \mathcal{A}(\bar{\xi}_{N-1})$, the output process satisfies 
\begin{align}
		X^{\bar{\xi}_{N-1}}_t&= G_i(X^{\bar{\xi}_{N-1}}_{T_i^{-}},\xi_{i})+\int_{T_i}^t \lambda_s(X^{\bar{\xi}_{N-1}}_{ s},\alpha_s)ds+\int_{T_i}^t \sigma_s(X^{\bar{\xi}_{N-1}}_s)dB^\alpha_s, \quad T_i \leq t < T_{i+1},\quad \PP^\alpha -a.s. \nonumber\\
        		 X^{\bar{\xi}_{N-1}}_t &= x_0+\int_0^t \lambda_s(X^{\bar{\xi}_{N-1}}_s,\alpha_s)ds+\int_0^t\sigma_s(X^{\bar{\xi}_{N-1}}_s)dB^\alpha_s, \quad 0 \leq t < T_{1}, \quad \PP^\alpha -a.s., \label{state0}\end{align}

where $i \in \{1,\ldots,N-1\}$, and
\begin{align*}
    \frac{d\PP^\alpha}{d\WW}:= \left(\int_0^T \sigma_s(X_s^{\bar{\xi}_{N-1}})^{-1}\lambda_s(X^{\bar{\xi}_{N-1}}_s,\alpha_s)dW_s \right),\quad 
    B_t^\alpha := W_t -\int_0^t\sigma_s(X_s^{\bar{\xi}_{N-1}})^{-1}\lambda_s(X^{\bar{\xi}_{N-1}}_s,\alpha_s)ds .  
\end{align*}

\begin{remark}
 It is worth commenting the economic interpretation of the state equation \eqref{state0}. The random variable $\Delta X^{\bar{\xi}_{N-1}}_{T_i}-\xi_{T_i}$ represents the execution cost that the firm experiences at each transaction time $T_i\in \mathcal{T}_N$. The action $\alpha \in \mathcal{A}$ indicates the level of effort that the agent makes affecting the distribution of the output process $X$. Finally, $\sigma(t,x)$ represents the risk level of the project when $X^{\bar{\xi}_{N-1}}_t = x$.
\end{remark}
    
\section{The principal-agent problem}\label{principal-agent}
In this section, we introduce a class of sequential optimal contracting problems that can be solved applying dynamic programming techniques. We start introducing the agent's objective. Firstly, we define the agent's utility function $U_a$, running cost $c$ and discounting process $k$ defined by the measurable functions: $U_a: E \mapsto \RR$, $c : [0,T]\times \RR^2\mapsto \RR$, and $k:[0,T]\times \RR^2 \mapsto \RR$ satisfying the following condition:
\begin{ass}\label{ass3}
The mappings $U_a$, $c$ and $k$ satisfy the following properties:
\begin{itemize}
\item[(i)] There exists a constant $K_u > 0$, for which
\begin{equation}\label{cond.utility}
    |U_a(y)| \leq K_u\left(1+ |y|\right), \quad \forall y \in E.
\end{equation}

\item[(ii)] For all $(t,x) \in [0,T]\times \RR$, 
 $c(t,x,\cdot) \in C^1(\RR)$, strictly convex and non negative. Moreover, there exist constants $K_c\geq 0$ and $p\geq 1$, for which
\begin{align*}
   \left|c(t,x,a)\right| \leq K_{c}\left(1+|a|^2+|x|\right), \quad \left| \frac{\partial}{\partial a}c(t,x,a) \right|\geq K_c|a|^{p}, \quad
   \lim_{|a|\rightarrow \infty}\frac{|c(t,x,a)|}{|a|} = \infty.
\end{align*} 
\item[(iii)] $k$ is bounded. Moreover, for all $(t,x) \in [0,T]\times \RR$, $k(t,x,\cdot)\in C^1(\RR)$, and  $\frac{\partial}{\partial a}k(t,x,\cdot)$ is bounded.
\end{itemize}
\end{ass}
Next, we introduce the controlled discount process defined for all $0 \leq t \leq s \leq T$
\begin{equation*}
    \mathcal{K}^\pi_{t,s} := \exp\left(-\int_t^s k_r(X_r,\alpha_r)dr \right).
\end{equation*}

Given a contract schedule $\bar{\xi}_N:=(\xi_1,\ldots,\xi_N) \in \Sigma_N$ with transaction times $\mathcal{T}_N := \{T_1,\ldots,T_N\}$, the agent maximizes the following objective: 
\begin{align}\label{ap:in}
    J_a \left(t,\alpha,\bar{\xi}_N\right) &
    :=\EE^{\PP^\alpha}\left[\mathcal{K}_{t,T}^\alpha U_a(\xi_N) +\sum_{i=1}^{N-1}\mathcal{K}^\alpha_{t,T_i}U_a\left(\xi_i\right)\mathbbm{1}_{t< T_i} -\int_{t}^{T}\mathcal{K}^\alpha_{t,s}c_s(X_s,\alpha_s)ds \bigg|\mathcal{F}_t\right],
\end{align}
where $(\alpha,X^{\bar{\xi}_{N-1}})$ solves \eqref{state0}.

Additionally, we define the agent's continuation value as follows:
\begin{equation*}
    V^a_t(\bar{\xi}_N) :=\esssup_{\alpha \in \mathcal{A}(\bar{\xi}_{N-1})}   J_a \left(t,\alpha ,\bar{\xi}_N\right).
\end{equation*}

Next, we introduce the principal's problem. Let $\mathcal{A}^*(\bar{\xi}_N)$ be the set of agent's optimal actions given the contract schedule $\bar{\xi}_N\in \Sigma_N$, the principal maximizes the following objective: 
\begin{align}\label{principalp}
    J_p\left(\bar{\xi}_N \right) &= \sup_{\alpha \in \mathcal{A}^*(\bar{\xi}_{N})}\EE^{\PP^{\alpha}}\left[U_p\left(l(X^{\bar{\xi}_{N-1}})-\xi_{N} \right)\right],
\end{align}
where $U_p$ is a real utility function, and $l : \Omega \mapsto \RR$ is a Borel measurable liquidation function.

Given an agent's reservation utility $R_a \in \RR$, we denote by $\Sigma^a_N$ the set of incentive compatible admissible contracts: 
\begin{equation*}
    \Sigma^a_N := \left\{\bar{\xi}_N\in {\Sigma}_N : V_a(\bar{\xi}_N)\geq R_a \right\}.
\end{equation*}
Hence, we define the principal's value as follows
\begin{equation*}
     V_p = \sup_{\bar{\xi}_N\in \Sigma^a_N}  J_p\left(\bar{\xi}_N \right).
\end{equation*}
In conclusion, the principal's problem is to find an incentive compatible contract schedule $\bar{\xi}_N \in \Sigma_N^a$, for which $J_p(\bar{\xi}_N) = V_p$. In order to represent the principal's problem effectively, we introduce the following class of processes.

For all $0 \leq t < s \leq T$, we define the following sets of processes:
  \begin{align*}
      \mathbb{H}([t,s]) &:=\left\{ Z,  \FF-\text{predictable } \bigg|\text{ } \EE^{\WW}\left[\left(\int_t^s \left|Z_r\right|^2dr \right)^{p/2}\right] < \infty, \forall p \geq 0\right\}, \\
      \mathbb{D}_{\text{exp}}([t,s]) &:= \left\{ Y, \FF-\text{predictable, càdlag } \bigg| \EE^{\WW}\left[\exp\left(p\sup_{t\leq r\leq s} \big|Y_r\big| \right)\right] < \infty, \forall p\geq 0\right\}.
  \end{align*}
 In Proposition \ref{lemma.SDE} we represent any admissible contract schedule in terms of a solution of a recursive system of BSDEs. We start by introducing the following functionals: 
\begin{align}  
h_t(a,x,y,z) &:= z\lambda_t(x,a) -k_t(x,a)y-c_t(x,a), \quad (t,a,x,y,z) \in [0,T]\times \RR^4, \nonumber \\
H_t(x,y,z) &:= \sup_{a\in \RR}h_t(a,x,y,z). \label{hamiltonian.1}
\end{align}
\begin{remark}
    Note that under assumptions $1$, $2$, the mapping $a \mapsto -h_t(x,y,z,a)$ is coercive implying that $H_t(x,y,z)$ is well defined for all $(t,x,y) \in [0,T]\times \RR^2$. Moreover, $z\mapsto H_t(x,y,z)$ is convex as it is the poinwise supremum of affine functions.
    \end{remark}

\begin{proposition}\label{lemma.SDE} Let $\bar{\xi}_{N}:=(\xi_1,\ldots,\xi_N) \in \Sigma_{N}$ be a contract schedule. Then, there exists a pair of processes  
    $(Z,Y)\in \mathbb{H}([0,T])\times \mathbb{D}_{\exp}([0,T])$, solving the following recursive system of BSDEs:
    \begin{align}
    &Y_t=Y_0-\int_{0}^{t}H_s\left(X_s^{\bar{\xi}_{N-1}},Y_s,Z_s\right)ds+\int_{0}^{t}Z_s\sigma_s(X_s^{\bar{\xi}_N})dW_s, \quad \WW -a.s.\quad 0\leq t < T_1\label{long.paymments} \\
    &Y_t=Y_{T_{i}^{-}}-U_a\left(\xi_{i} \right)-\int_{T_{i}}^{t}H_s\left(X_s^{\bar{\xi}_{N-1}},Y_s,Z_s\right)ds+\int_{T_{i}}^{t}Z_s\sigma_s(X^{\bar{\xi}_{N-1}}_s)dW_s, \quad \WW-a.s.\nonumber\quad T_i \leq t < T_{i+1},
    \end{align}
    for all $i \in \{1,\ldots,N-1\}$.
\end{proposition}
Using the previous recursive system of BSDEs we characterize the agent's optimal response for any contract schedule $\bar{\xi}_N \in \Sigma_N$ satisfying an integrability constraint. 
\begin{proposition}  \label{BSDE.value.agent}
Let $\bar{\xi}_N \in \Sigma_N $, and  $(Y,Z) \in   \mathbb{D}_{\text{exp}}([0,T])\times \mathbb{H}([0,T]) $ be the unique pair of processes defined in Proposition \ref{lemma.SDE}. Assume $\EE\left[\mathcal{E}\left(\int_0^TZ_sdW_s\right)^{1+\epsilon}\right] < \infty$, for some $\epsilon >0$. Then, 
\begin{align*}
    Y_t = V_t^a(\bar{\xi}_N) \hspace{2mm} \WW-a.s., \quad 0 \leq t \leq T.
\end{align*}

Moreover,  $\alpha \in \mathcal{A}^*(\bar{\xi}_N)$ if and only if     
\begin{align}
    \alpha_t &= \hat{\alpha}_t(X_t^{\bar{\xi}_{N-1}},Y_t,Z_t), \quad dt\otimes \WW -a.e., \nonumber\\
    \hat{\alpha}_t(x,y,z) &\in \argmax_{} h_t(\cdot,x,y,z), \quad \forall (t,x,y,z) \in [0,T]\times \RR^3. 
    \label{alphasup}
\end{align}
\end{proposition}
\begin{remark}
The previous proposition motivates us to restrict the set of admissible contract schedules excluding contracts that are not integrable enough. The later does not suppose a significant limitation and allows us to fully characterize the set of agent's optimal responses.    
\end{remark}
We introduce the following class of contract schedules: 
\begin{align*}
    \hat{\Sigma}^a_N &:= \left\{\bar{\xi}_N \in \Sigma^a_N : (Y,Z)   \text{ solving }  \eqref{BSDE.value.agent}, \hspace{2mm}\EE\left[\mathcal{E}\left(\int_0^TZ_sdW_s\right)^{1+\epsilon}\right] < \infty, \hspace{2mm} \text{ for some } \epsilon >0  \right\}. 
\end{align*}
In addition, for any $\bar{\xi}_N \in \hat{\Sigma}^a_N $ we denote by $\hat{\mathcal{A}}(\xi_N)$ the set of $\FF$-adapted processes $\alpha$ satisfying
\begin{align}
     \alpha_t &= \hat{\alpha}_t(X_t^{\bar{\xi}_{N-1}},Y_t,Z_t), \quad dt\otimes \WW -a.e., \nonumber\\
    \hat{\alpha}_t(x,y,z) &\in \argmax_{} h_t(\cdot,x,y,z), \quad \forall (t,x,y,z) \in [0,T]\times \RR^3. 
\end{align}
Finally, we reduce the principal's problem to a weak formulation stochastic control problem. 
\begin{theorem}\label{theoremstochcontrol}The principal's value satisfies
\begin{align*}
    \sup_{\bar{\xi}_N \in \hat{\Sigma}_N^a}J_p(\bar{\xi}_N) =\sup_{Y_0 \geq R_a}\sup_{(Z,\bar{\xi}_{N-1})\in {\mathbb{H}}([0,T])\times \Sigma_{N-1}}\sup_{\alpha \in \hat{\mathcal{A}}(\bar{\xi}_{N-1})} \EE^{\PP^{\alpha}} \left[U_p\left(l\left(X^{\bar{\xi}_{N-1}}\right)-U_a^{-1}\left(Y^{\bar{\xi}_{N-1},Z}_T\right)\right)\right],
\end{align*}
where $(X^{\bar{\xi}_{N-1}},Y^{\bar{\xi}_{N-1},Z})$ solves the following iterative system of SDEs:
\begin{align*}
    &Y_t = Y_{T_i^{-}}-U_a(\xi_i)-\int_{T_i}^tH_s(X_s,Y_s,Z_s)ds+\int_{T_i}^tZ_s dX^{\bar{\xi}_{N-1}}_s, \quad T_i \leq t < T_{i+1},\\
    &Y_t = Y_0-\int_{0}^tH_s(X_s^{\bar{\xi}_{N-1}},Y_s,Z_s)ds+\int_{0}^tZ_sdX^{\bar{\xi}_{N-1}}_s, \quad 0\leq t < T_1,\\
    &X^{\bar{\xi}_{N-1}}_t = G_i\left(X^{\bar{\xi}_{N-1}}_{T_i^{-}},\xi_i\right)+\int_{T_i}^t\left(\lambda_s\left(X_s^{\bar{\xi}_{N-1}},\hat{\alpha}_s(X^{\bar{\xi}_{N-1}}_s,Z_s,Y_s)\right)ds+\sigma_s(X^{\bar{\xi}_{N-1}}_s)dB^\alpha_s \right) ,\quad T_i\leq t < T_{i+1},  \\
    &X^{\bar{\xi}_{N-1}}_t = x_0 +\int_{0}^t \left(\lambda_s\left(X^{\bar{\xi}_{N-1}},\hat{\alpha}_s(X^{\bar{\xi}_{N-1}}_s,Z_s,Y_s)\right)ds+\sigma_s(X^{\bar{\xi}_{N-1}}_s)dB^\alpha_s\right), \quad 0 \leq t < T_1, \quad \PP^{\alpha} -a.s.
\end{align*}
 
\end{theorem}
\section{The benchmark model
}\label{benchmark}
In this section we use the results obtained in the previous sections and apply it to a benchmark model with a risk neutral principal and risk averse, time sensitive agent. We assume that the principal compensates the agent with $N$ payments $\bar{\xi}_N := (\xi_{1},\ldots,\xi_{N})$ transacted at $\mathcal{T}_N:=\{T_1,\ldots,T_{N}\}$, where $(T_i)_{i=0}^N$ is a strictly increasing sequence satisfying $T_0 = 0, T_N= T$. Given a contract schedule $\bar{\xi}_N \in \Sigma_N$, and an initial output level $x_0 \in \RR$, we introduce the output process:
\begin{align*}
    X^{\bar{\xi}_{N-1}}_t := x_0 + W_t -\sum_{j=1}^{N-1}\xi_j \mathbbm{1}_{T_j \leq t}, \quad \WW -a.s.
\end{align*}
We denote by $\mathcal{A}$ the set of admissible agent's actions. The previous corresponds to the set of $\FF$ adapted processes $\alpha$ satisfying
\begin{align*}
    &\EE\left[\exp \left(-\frac{1}{2}\int_0^T\alpha^2_sds+\int_0^T\alpha_sdW_s \right)^{1+\epsilon}\right]< \infty,
\end{align*}
for some positive constant $\epsilon >0$.

 Hence, given an agent's action $\alpha \in \mathcal{A}$, and a contract schedule $\bar{\xi}_N \in \Sigma_N$, the output process satisfies the following controlled dynamics:
\begin{align*}
&X_t = x_0-\sum_{j=1}^{N-1} \xi_{j}\mathbbm{1}_{T_j\leq t} +\int_0^t\alpha_sds +  B^\alpha_t,  \quad\PP^\alpha-a.s., \quad 0 \leq t \leq T, 
\end{align*}
where
\begin{align*}
    B^\alpha_t := W_t - \int_0^t \alpha_r dr, \quad
    \frac{d\PP^\alpha}{d\WW}:= \exp \left(-\frac{1}{2}\int_0^T\alpha^2_sds+\int_0^T\alpha_sdW_s \right).
\end{align*}

In this scenario, we introduce the agent's objective and continuation value as follows
\begin{align*}
    J_a(t,\alpha,\bar{\xi}_N) &:=\EE^{\PP^\alpha}\left[\sum_{i=1}^N e^{-k_a(T_i-t)}U_a(\xi_{i})\mathbbm{1}_{T_i>t}-\frac{1}{2}\int_0^T e^{-{k_a}(s-t)}\alpha^2_sds \bigg| \mathcal{F}_t\right],\\
    V_t^a(\bar{\xi}_N) &:=\esssup_{\alpha \in \mathcal{A}} \EE^{\PP^\alpha}\left[\sum_{i=1}^N e^{-k_a(T_i-t)}U_a(\xi_{i})\mathbbm{1}_{T_i>t}-\frac{1}{2}\int_0^T e^{-{k_a}(s-t)}\alpha^2_sds \bigg| \mathcal{F}_t\right]
\end{align*}
where, $U_a(y) := y^{1/\gamma}\mathbbm{1}_{y\geq 0}-\infty \mathbbm{1}_{y<0}$, for some $\gamma > 1$, and $k_a\geq0$ represents the agent's discounting factor.

On the other hand, we introduce the principal's objective and continuation value as follows
\begin{align}
    J_p\left(\bar{\xi}_N\right) := \EE^{\PP^\alpha}\left[X_T-\xi_N \right],
    \quad V_p := \sup_{\bar{\xi}_N \in \Sigma^a_N}\sup_{\alpha\in \mathcal{A}}J_p\left(\bar{\xi}_N\right),\label{principal.problem1}
\end{align}
where $\Sigma^a_N := \left\{ \bar{\xi}_N \in \Sigma_N \big| V_a(\bar{\xi}_N)\geq R_a, \hspace{1 mm} V_t^a(\bar{\xi}_N) \geq 0, \forall t \in [0,T]\right\}$.

In our formulation $R_a \geq 0$ represents the agent's reservation utility level. Additionally, we assume that the agent has limited liability implying that the agent only accepts contracts that generate a non-negative continuation utility at any time. 
 
 Applying Proposition \ref{BSDE.value.agent}, we obtain that for any contract schedule $\bar{\xi}_N \in \Sigma^a_N$ there exists a unique pair of processes $(Y,Z) \in \mathbb{H}([0,T])\times \mathbb{D}_{\text{exp}}([0,T])$ satisfying

    \begin{align*}
    Y_t &= Y_{T_{i}^{-}}-U_a(\xi_i) +\int_{T_i}^t\left(  \frac{1}{2}Z_s^2 + k_a Y_s\right)ds+\int_{T_i}^tZ_s dW_s,\quad T_i \leq t < T_{i+1},\quad 1\leq i< N-1, \\
    Y_t &= Y_{0}+\int_{0}^t \left( \frac{1}{2}Z_s^2 + k_a Y_s\right) ds+\int_{0}^tZ_s dW_s,\quad 0 \leq t < T_1,
    \end{align*}
    where $Y_0 \in \RR$.

Next, we invoke (\cite{ElKarouiTan2024},Theorem $4.3$) to justify that the previous weak formulation stochastic control problem is equivalent to a stochastic control problem in the strong formulation where the stochastic basis is fixed.  
Using Theorem \ref{theoremstochcontrol} and the previous observation, we reduce the principal's problem \eqref{principal.problem1}  to the following strong formulation stochastic control problem: 
\begin{equation}\label{principal.value.benchmark}
    V(t,x,y) : = \sup_{Y_0 \geq R_a}\sup_{\bar{\xi}_N \in \Sigma_N(t,T)}\sup_{(Z,\bar{\xi}_{N-1}) \in \mathcal{U}(t,T)}\EE\left[X^{t,x,\bar{\xi}_{N-1}}_T-\left(Y^{t,y,\bar{\xi}_{N-1}}_T\right)^{\gamma}  \right],
\end{equation}
where, for $s>t$:
\begin{align*}
     &X^{t,x,\bar{\xi}_{N-1},Z}_s = x-\sum_{j=1}^{N-1} \xi_{j}\mathbbm{1}_{t<T_j\leq s} +\int_t^sZ_rdr +  B_t-B_s, \\
     &Y^{t,x,\bar{\xi}_{N-1},Z}_s = y-\sum_{j=1}^{N-1} U_a\left(\xi_{j}\right)\mathbbm{1}_{t<T_j\leq s} +\int_t^s\left(\frac{1}{2}Z^2_r+k_aY_r\right)dr + \int_t^sZ_rdB_r, 
\end{align*}
where $B$ is a brownian defined on fixed probability space $(\Omega,\FF,\PP)$, and $\FF$ is the augmented filtration generated by $B$. The control $Z$ is an $\FF$-predictable process and $\bar{\xi}_{N-1}:= (\xi_1,\ldots,\xi_{N-1})$ where, for every $i \in \{1,\ldots,N-1\}$, $\xi_i$ is a $\mathcal{F}_{T_i}$-measurable random variable satisfying the integrability stated in \eqref{set.contracts}. Moreover, the set $\mathcal{U}$ denotes the set of controls $(Z,\bar{\xi}_{N-1})\in \mathbb{H}([0,T])\times \Sigma_{N-1}$ satisfying the limited liability condition: $Y_{t}^{0,y_0,Z,\bar{\xi}_{N-1}}\geq 0$ $\PP-a.s.$, for all $(y_0,t) \in [R_a,\infty)\times [0,T]$, and the uniform boundedness condition: $|Z_t|\leq K, dt\otimes \WW -a.e.$
\begin{remark}
    From a modeling perspective, the restriction to uniformly bounded continuous controls $|Z_t| \leq K dt\otimes \WW -a.e.$, establishes a maximum sensitivity level of the agent's value with respect to changes in the output process. Jointly with the limited liability constraint and the reservation utility $R_a$ comprise the agent's participation constraint. 
\end{remark}
The following theorem allows us to write the principal's value function as the unique viscosity solution of a recursive system of HJB equations.

\begin{theorem}\label{hjb:pp:in}
    The principal's value function satisfies the following properties:
    \begin{enumerate}
             \item  The value function satisfies $V(t,x,y)= x+v(t,y)$, where $v$ is the unique viscosity solution of the following constrained HJB equation:
        \begin{align}\label{constrainedHJB}
             &v_t+G\left(t, y, v, v_y,  v_{yy}\right)=0, \quad (t,y) \in [T_{i-1}, T_{i})    \times (0,\infty),\\
             &v(T^{-}_{i}, y)=f_{i}(y), \quad y \geq 0, \quad i \in \{1, \cdots ,N-1\} \nonumber, \\
            &v(T_{N}, y)=-y^\gamma, \quad y \in [0,\infty], \nonumber \\
            &v(t, 0)=0, \quad  t \in [T_{i-1},T_{i}), \nonumber
        \end{align}
        where 
        $G(t,y,v_y,v_{yy}) := k_a yv_y + \sup_{|z|\leq K}\left\{z+\frac{1}{2}\left(v_y+v_{yy} \right)z^2\right\}$, and $f_i(y):= \max_{0 \leq \eta \leq y} -\eta^\gamma+v(T_i,y-\eta)$. Moreover, $v$ is continuous in each region $\mathcal{R}_i:=[T_{i-1},T_{i})\times [0,\infty)$, $ i \in \left\{1 \ldots  N\right\}$.
        \item The function $f_i$ is continuous and has polynomial growth for all $i \in \{1,\ldots,N-1\}$. Moreover, there exits a minimal function $\eta^*_i:[0,\infty) \mapsto \RR$ satisfying 
        \begin{itemize}
            \item $\eta_i^*(y) \in \argmax_{0\leq \eta \leq y} v(T_i,y-\eta)-\eta^\gamma$, for all $y \geq 0$, $i \in \{1,\ldots,N-1\}$. 
            \item $\eta^*_i$ is lower semicontinuous, for all $y\geq 0$, $i \in \{1,\ldots,N-1 \}$.
        \end{itemize}
        \item  $V(T_i^{-},x,y) = x+f_i(y)$, for all $(t,x,y) \in [0,T]\times \RR \times [0,\infty)$, $i \in \{1,\ldots N-1\}$.
            \end{enumerate}
\end{theorem}

\section{Numerical results}\label{numerics}
We employ numerical methods to discuss the model presented in the previous section. In the analysis, we fix $U_a(y) = y^{1/2}$ and the recursive HJB equation \eqref{constrainedHJB} is approximated numerically following a finite differences scheme. We refer to \cite{Convergenceviscositysougonathis} to justify the convergence of our numerical scheme to the principal's value function. 

\subsection{The case with no discounting  $\kappa_a=0$}

\begin{figure}[H]
     \centering
     \begin{subfigure}[b]{0.47\textwidth}
         \centering
         \includegraphics[width=\textwidth]{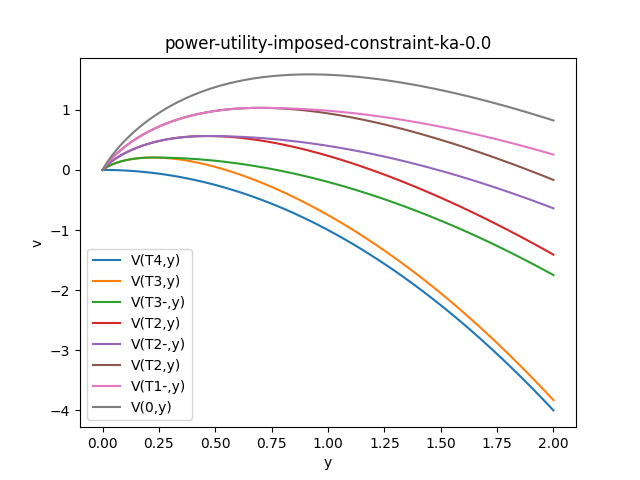}
         \caption{Principal's value with $k_a =0$.}
         \label{fig:ka=0,vy}
     \end{subfigure}
     \quad
     \begin{subfigure}[b]{0.47\textwidth}
         \centering
         \includegraphics[width=\textwidth]{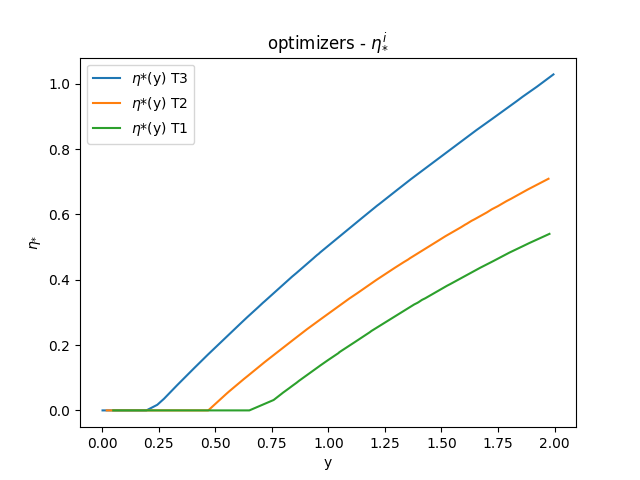}
         \caption{Static optimizers with $k_a=0$.}
         \label{fig:ka=0,eta}
     \end{subfigure}
        \caption{Initial Negotiation $k_a = 0$}
        \label{fig:pv,soka=0}
\end{figure}
Figure \eqref{fig:ka=0,vy} illustrates the principal's value function at each transaction time when $k_a =0$. The functions  \(v(T_{i},\cdot)\) represent the principal's value function at time \(T_{i}\), and \(v(T^{-}_{i},y)\) is the principal's value function immediately before \(T_{i}\). On the other hand, figure \eqref{fig:ka=0,eta} depicts the utilities that the principal provides to the agent at each intermediate transaction time. The function \(\eta_i^{*}(\cdot):=U_a(\xi_i^*(\cdot)) \) denotes the optimal utility the principal aims to pay to the agent at $T_i$, $i \in \{1,\ldots, N-1\}$.
We see that the principal's value function in figure \eqref{fig:ka=0,vy} exhibits concavity and it is ultimately decreasing with respect to the agent's continuation utility in a fixed time. The later is consistent with Sannikov's model (Lemma 8.1 \cite{possamai2024there}). Moreover, our numerical results indicate that an agent with small reservation utility has informational rent, meaning that the principal optimally offers a contract providing to the agent a utility strictly greater than the agent's reservation utility.

Secondly, figure \eqref{fig:ka=0,eta} shows that the optimal intermediate payments are an increasing function of the agent's utility at a given time. Moreover, for the same agent's utility level, the principal offers more utility to the agent in a posterior payment. Finally, \eqref{fig:ka=0,eta} reflects the fact that the region where the agent has informational rent shrinks as time elapses and the principal only compensates the agent when his utility does not belong to the informational rent region.

\subsection{The effect of the discounting factor $k_a>0$}
\begin{figure}[H]
     \centering
     \begin{subfigure}[b]{0.4\textwidth}
         \centering
         \includegraphics[width=\textwidth]{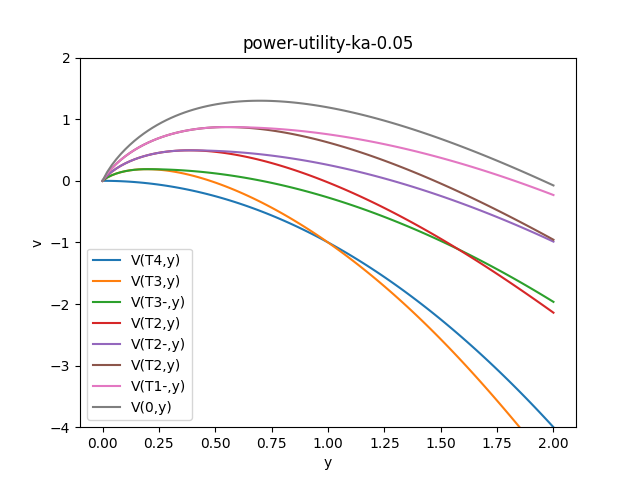}
         \caption{Principal's value with ka = 0.05 }
         \label{fig:ka=0.05,vy}
     \end{subfigure}
     \begin{subfigure}[b]{0.4\textwidth}
         \centering
         \includegraphics[width=\textwidth]{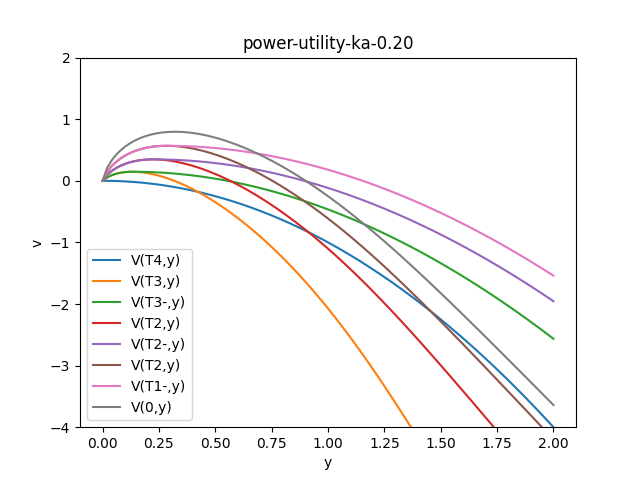}
         \caption{Principal's value with ka = 0.05}
         \label{fig:ka=0.2,vy}
     \end{subfigure}
        \caption{Discounting Comparison (Principal's Value Function)}
        \label{discouning_initial_negotiation}
\end{figure}

The figure above represents the principal's value function at each transaction time when the agent is inpatient: $k_a>0$. 
The above simulations align with the concept of employment interval discussed in \cite{sannikov2008continuous}. In our case, it directly shows that it is harmful for the principal to provide a large utility to the agent when the agent's discounting factor is greater than zero.
We define the $i$-{th} employment interval as the set $\mathcal{E}_i:=\{y \in [0,\infty): V(T_{i-1},y) \geq V(T_{i},y)\}$. A careful observation of the above figures shows that the employment interval shrinks when the time elapses. 
Moreover, a comparison between figures \eqref{fig:ka=0.05,vy} and \eqref{fig:ka=0.2,vy} shows that an increase of the discounting factor causes an shrinkage of the employment intervals in every contracting period.

\begin{figure}[H]
     \centering
     \begin{subfigure}[b]{0.47\textwidth}
         \centering
         \includegraphics[width=\textwidth]{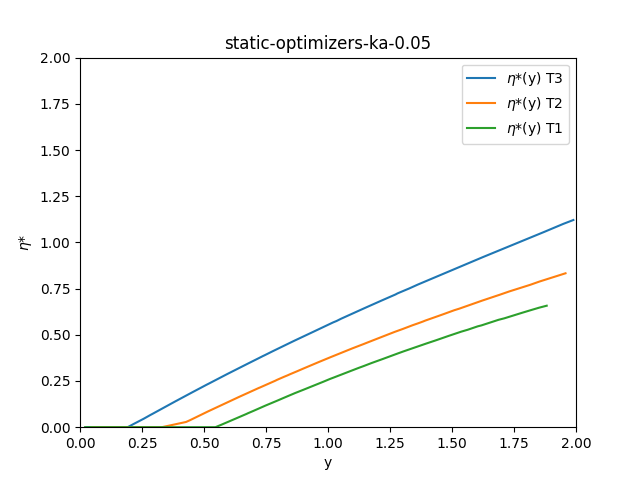}
         \caption{Optimal utilities delivered to the agent. $k_a = 0.05$}
         \label{fig:so_ka=0.05,eta}
     \end{subfigure}
     \quad
     \begin{subfigure}[b]{0.47\textwidth}
         \centering
         \includegraphics[width=\textwidth]{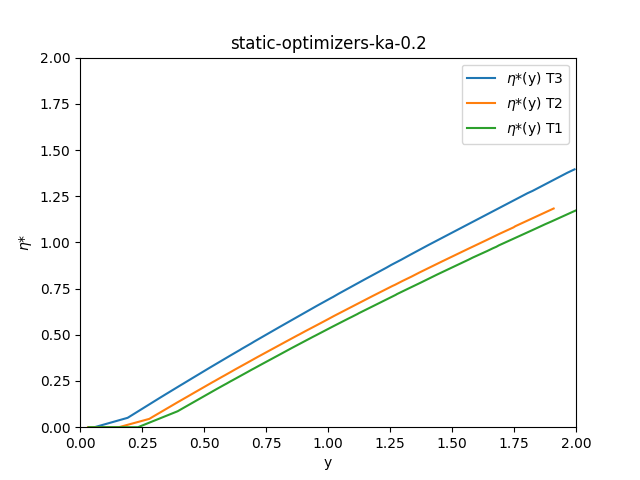}
         \caption{Optimal utilities delivered to the agent. $k_a = 0.2$}
         \label{fig:so_ka=0.2,eta}
     \end{subfigure}
        \caption{ Optimal utilities $\eta_i^*=U_a(\xi_i^*)$.}
        \label{fig:sokas}
\end{figure}
Figure \eqref{fig:sokas} above describes the optimal utilities delivered to the agent associated to the principal's value functions represented in figure \eqref{discouning_initial_negotiation}. We see that, with an increase in the discounting factor, the truncation region \(\{y : \eta^{*}_{i}(y) = 0\}\) shrinks. This means the principal prefers to provide more utilities to an agent with a high discounting factor. Finally, a careful comparison among figures  \eqref{fig:so_ka=0.05,eta} and \eqref{fig:so_ka=0.2,eta} shows that it is not optimal for the principal to fully compensate the agent before the terminal time. 

\subsection{Analysis of the payments' distribution}

\begin{figure}[H]
     \centering
     \begin{subfigure}[b]{0.47\textwidth}
         \centering
         \includegraphics[width=\textwidth]{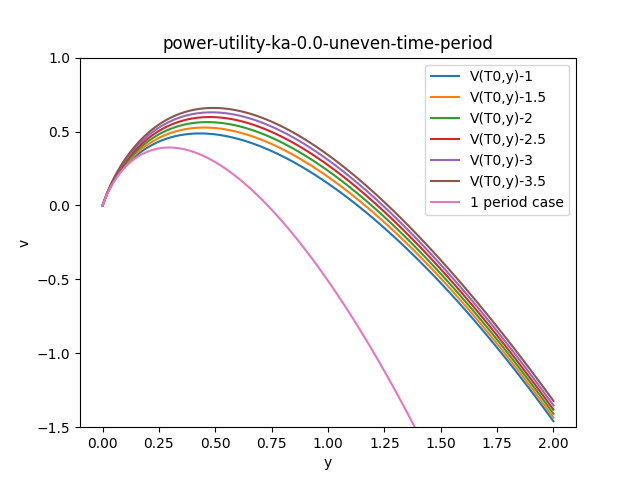}
         \caption{Principal's value $V(T_{0},y)$, $k_a = 0$}
         \label{fig:pvunevenska=0.0v0y}
     \end{subfigure}
     \quad
     \begin{subfigure}[b]{0.47\textwidth}
         \centering
         \includegraphics[width=\textwidth]{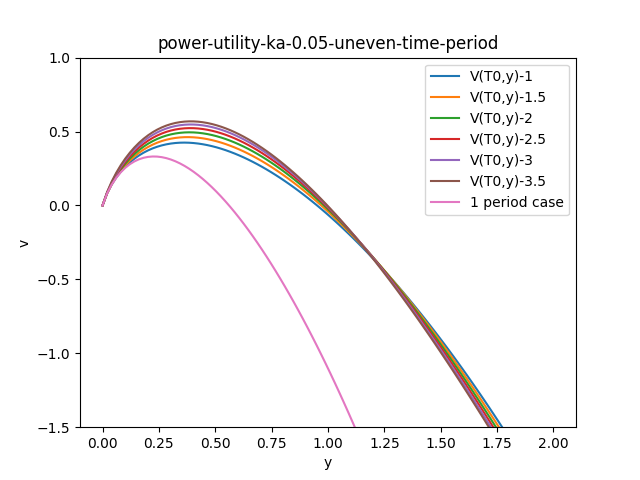}
         \caption{Principal's value $V(T_{0},y)$, $k_a = 0.05$}
         \label{fig:pvunevenska=0.05v0y}
     \end{subfigure}
        \caption{Uneven contracting periods.}
        \label{fig:distribution_contracting}
\end{figure}

Figure \eqref{fig:distribution_contracting} illustrates the problem with two payments with a terminal contract transacted at a fixed time $T=4$, and a flexible initial payment transacted at one of the following times: $T_1 = 1+\frac{i}{2}$, $i\in \{0,1,2,3,4,5\}$. For example, '$V(T_{0}, y) - 1.5$' represents the principal's value function at $T_0 = 0$ when the first payment occurs at $T_{1} = 1.5$ and the second payment occurs at $T_{2} = 4$. Additionally, the '1 period case' refers to the classic principal-agent model with a single payment transacted at the end of the contracting period $T=4$.

Firstly, we observe that the principal benefits from making multiple payments  compared to a single terminal compensation. Secondly, when $k_a =0$, the principal always benefits from delaying the first payment. Finally, subfigure \eqref{fig:pvunevenska=0.05v0y} depicts the principal's value functions when the agent is impatient ($k_a = 0.05$). It shows that if the agent is impatient with a relatively high reservation utility, the principal benefits from setting the first payment earlier. Conversely, if the agent is impatient with a low reservation utility, the principal benefits from delaying the first payment.
\subsection{Analysis of the payment frequency}
\begin{figure}[H]
    \centering
    \includegraphics[width=0.45\linewidth]{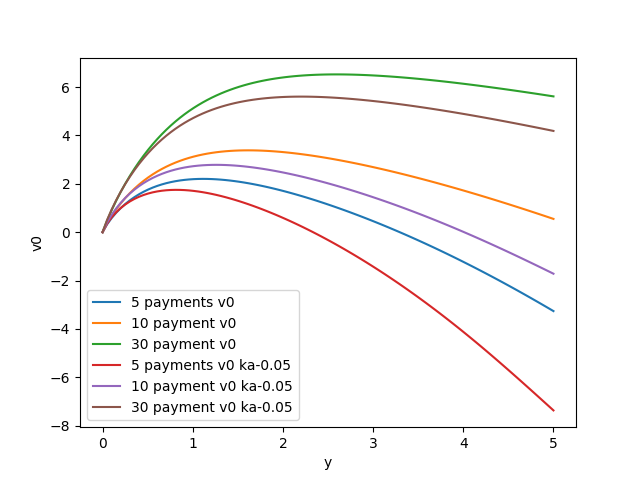}
    \caption{Payment Frequency Comparison}
    \label{fig:frequency}
\end{figure}
Figure \eqref{fig:frequency} illustrates the principal´s value $V(0,\cdot)$ for different payment frequencies. For $N=5,10, 30$, we consider a sequential optimal contracting problem with $N$ contracts transacted at \(T_{i} = i\frac{T}{N}\), $ i \in \{1,\ldots,N\}$, and $T =10$. We observe that as the number of payments in the contract schedule increases, the principal achieves higher profits for any given agent's reservation utility. The previous is related to the absence of transaction costs in our model. 
\subsection{Initial Negotiation VS Renegotiation}
In this subsection we compare the principal's value in our benchmark model (initial negotiation) with an analogous model assuming that each contract is renegotiated right after the previous payment is transacted. In the renegotiation setting the agent solves a different optimization problem in each contracting region. Moreover, the agent's continuation utility resets when the $i$-th payment is delivered. Mathematically, the renegotiation problem corresponds to a sequential Stackelberg game with $N$ different reservation utility constraints:
\begin{align}\label{reservation.utility.agent.reneg}
    \sup_{\alpha \in \mathcal{A}}\EE\left[e^{-k_aT_i}U_a(\xi_i) -\int_{T_{i-1}}^{T_i}e^{-k_as}\alpha^2_sds \right]\geq e^{-k_aT_{i-1}}R_a^i, \quad i \in \{1, \ldots, N\},
\end{align}
where $R_a^i$ is the reservation utility level in the $i$-th contracting period. 

On the other hand, recall the definition of the reservation constraint in the initial negotiation problem: 
\begin{align*}
    \sup_{\alpha \in \mathcal{A}}\EE\left[\sum_{i=1}^Ne^{-k_aT_i}U_a(\xi_i) -\int_{0}^{T}e^{-k_as}\alpha^2_sds  \right]\geq R_a,
\end{align*}
where $R_a$ is the agent's reservation utility in the initial negotiation problem. 

Assuming that the agent allocates his utility uniformly over the duration of the contract, we obtain normalizing the units in \eqref{reservation.utility.agent.reneg}: 
\begin{equation}\label{reservation.reneg}
R_a^i =e^{k_aT_{i-1}}\frac{(T_i-T_{i-1})}{T}R_a, \quad  i \in \{1,\ldots,N\}.
\end{equation}
The following figures represent the principal's value function in both settings assuming $N=4$,$T=8$,  $T_i =2i$, $i \in \{0,\ldots,4\}$.

\begin{figure}[H]
     \centering
     \begin{subfigure}[b]{0.47\textwidth}
         \centering
         \includegraphics[width=\textwidth]{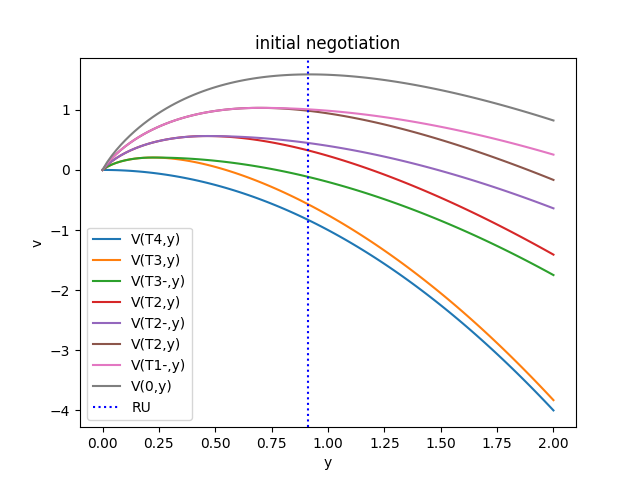}
         \caption{Principal's value. Initial negotiation. $k_a = 0$.}
         \label{fig:ka=0,vyru}
     \end{subfigure}
     \quad
     \begin{subfigure}[b]{0.47\textwidth}
         \centering
         \includegraphics[width=\textwidth]{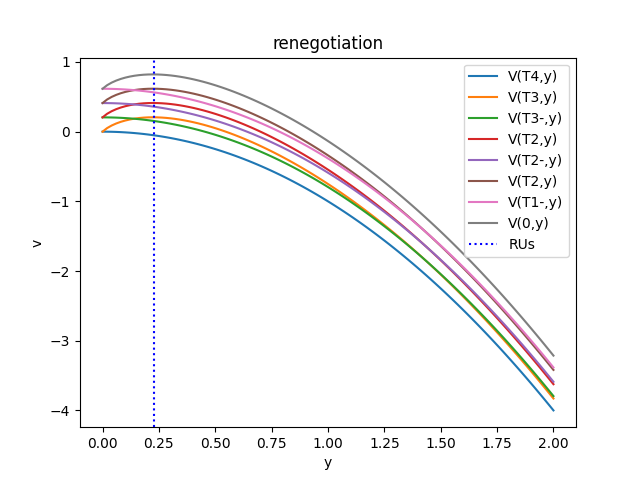}
         \caption{Principal's value. Renegotiation. $k_a = 0$.}
         \label{fig:renegotiation}
     \end{subfigure}
        \caption{Initial Negotiation VS Renegotiation $k_a = 0$}
        \label{fig:IvsReka0.0}
\end{figure}
Figure \eqref{fig:IvsReka0.0} describes the principal's value in the initial negotiation and renegotiation settings. In both cases we fix  $k_a = 0$, and $R_a =0.909$. The utilities of each contracting period in the renegotiation problem are computed using \eqref{reservation.reneg}. We observe that the principal benefits from an initial negotiation compared to a renegotiation setting.

\begin{figure}[H]
     \centering
     \begin{subfigure}[b]{0.44\textwidth}
         \centering
         \includegraphics[width=\textwidth]{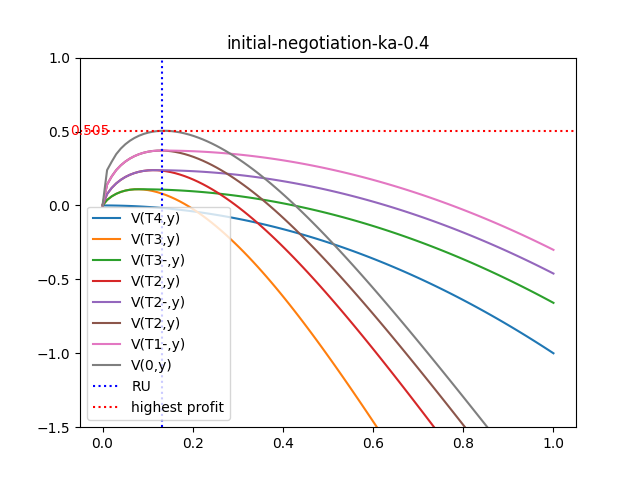}
         \caption{Principal's value function. Initial negotiation, $k_a =0.4$. }
         \label{fig:ka=0.4,vyruinbest}
     \end{subfigure}
     \quad
     \begin{subfigure}[b]{0.44\textwidth}
         \centering
         \includegraphics[width=\textwidth]{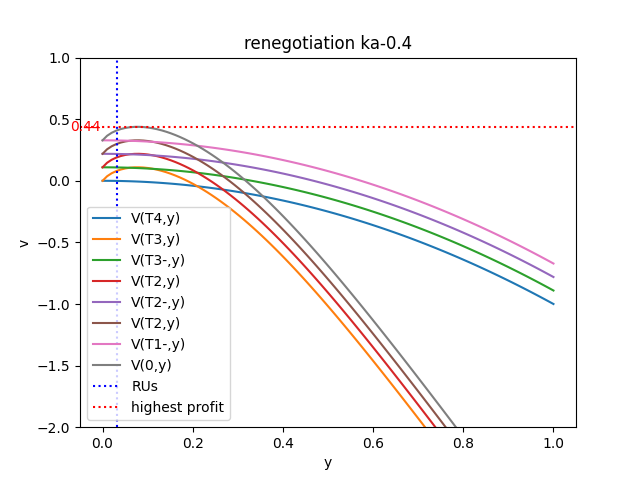}
         \caption{Principal's value function. Renegotiation, $k_a = 0.4$.}
         \label{fig:renegotiation_ka_0.4_inbest}
     \end{subfigure}
        \caption{Initial Negotiation VS Renegotiation, $k_a = 0.4$.}
        \label{fig:IvsRka0.4inbest}
\end{figure}
Figure  \eqref{fig:IvsRka0.4inbest} describes the principal's value in the initial negotiation and renegotiation settings. In both cases we fix  $k_a = 0.4$, and $R_a = 0.131$. The reservation utilities in the renegotiation problem are computed using \eqref{reservation.reneg}. We observe that the principal benefits again from an initial negotiation compared to a renegotiation setting.

\begin{figure}[H]
     \centering
     \begin{subfigure}[b]{0.47\textwidth}
         \centering
         \includegraphics[width=\textwidth]{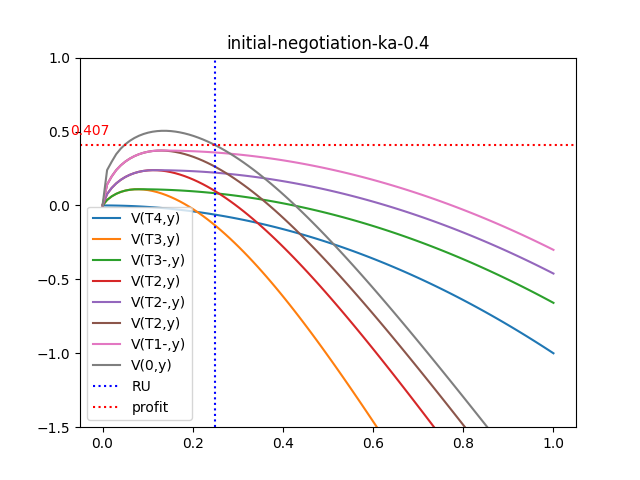}
         \caption{Principal's value function. Initial negotiation.  $k_a = 0.4$.}
         \label{fig:ka=0.3,vyru}
     \end{subfigure}
     \quad
     \begin{subfigure}[b]{0.47\textwidth}
         \centering
         \includegraphics[width=\textwidth]{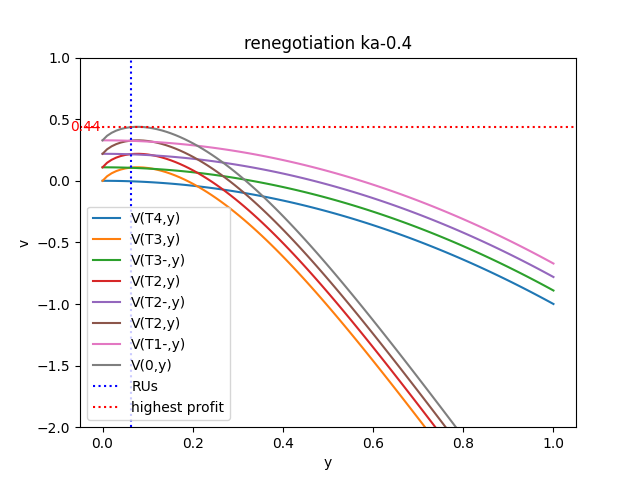}
         \caption{Principal's value function. Renegotiation. $k_a = 0.4$.}
         \label{fig:renegotiation_ka_0.3}
     \end{subfigure}
        \caption{Initial Negotiation VS Renegotiation, $k_a = 0.4$.}
        \label{fig:IvsRka0.4in}
\end{figure}
Figure  \eqref{fig:IvsRka0.4in} describes the principal's value in the initial negotiation and renegotiation settings. In both cases we fix  $k_a = 0.4$, and $R_a =0.25$. The reservation utilities in the renegotiation problem are computed using \eqref{reservation.reneg}. We observe that the principal benefits from renegotiation. The later shows that the principal's optimal negotiation setting (initial or renegotiation) depends on the contracting environment.

\appendix

\section{Appendix}

{\it{Proof of Proposition \ref{lemma.SDE}}:}
\\

Let $\bar{\xi}_N \in \Sigma_N$. Using Assumptions \ref{ass1}, \ref{ass2}, \ref{ass3} and the uniform boundedness of $k$, there exits $B_1>0$, for which
\begin{align}
        &H_t(x,y,\cdot) \text{ is convex,} \quad \forall (t,x,y) \in [0,T]\times \RR^2,\nonumber\\
    &|H_t(x,y_1,z)-H_t(x,y_2,z)| \leq B_1|y_1-y_2|, \quad \forall (t,x,z,y_1,y_2) \in [0,T]\times \RR^4. \label{hamiltonian.bound}
\end{align}
Moreover, there exists $B_2>0$, for which
\begin{align}
    \left|H_t(x,y,z,a)\right| = \left|\sup_{a\in \RR}\left\{\lambda(t,x,a)z -c(t,x,a)+yk(t,x,a)\right\} \right|\leq \left|\sup_{a\in \RR}\left\{\lambda(t,x,a)z -c(t,x,a)\right\}\right| + B_2|y|. \label{bound.hamiltonian.lemma1}
\end{align}
Next, we observe that for all $(t,x)\in [0,T]\times \RR$, the mapping $a \mapsto \lambda(t,x,a)z -c(t,x,a)$,
admits a global maximizer $a^*(t,x)$ satisfying the first order condition:
\begin{align*}
   \frac{\partial }{\partial a} \lambda(t,x,a^*(t,x))z = \frac{\partial }{\partial a} c(t,x,a^*(t,x)).
\end{align*}
Hence, using assumption \ref{ass3}, there exists $C>0$ such that
\begin{align*}
    |a^*(t,x)| \leq C \left(1+|z|\right).
\end{align*}
Plugging in the last expression into \eqref{bound.hamiltonian.lemma1}, we obtain 
\begin{equation}
    |H_t(x,y,z)| \leq B_2|y| +C |z|^2.
\end{equation}
  Applying (\cite{briandquadraticbsdeconvexgenerators2007}, Corollary 2), there exists a unique pair of processes $(Y^{N},Z^{N}) \in \mathbb{D}_{\text{exp}}([0,T]) \times \mathbb{H}([0,T])$ solving the BSDE: 
\begin{align*}
    dY^N_t &= H_t\left(X^{\bar{\xi}_{N}}_t,Y^N_t,Z^N_t\right)dt-Z^N_t\sigma_t(X_t^{\bar{\xi}_N})dW_t,  \\
    Y^N_T &= U_a\left(\xi_N\right).
\end{align*}
Therefore, for all $p\geq 0$, there exits $C_p, p'>0$, such that
$$\EE\left[\exp\left(p\left|U_a(\xi_{N-1})+Y^N_{T_{N-1}}\right|\right)\right] < C_p\EE\left[\exp\left(p'\left|\xi_{N-1}\right|\right)\right]^{1/2}\EE\left[\exp\left(2p\left|Y^N_{T_{N-1}}\right|\right)\right]^{1/2} < \infty,$$ 
where we used Cauchy-Schwarz inequality, and $\left(\bar{\xi}_N,Y\right) \in \Sigma_N\times \mathbb{D}_{\text{exp}}\left(\left[0,T\right]\right)$.
\\
Repeating the same argument we show by induction that for all $i \in \{1,\ldots,N-1\}$ there is a unique pair of processes $(Y^i,Z^i) \in \mathbb{D}_{\text{exp}}([0,T_i])\times \mathbb{H}([0,T_i])$ satisfying 
\begin{align}
    dY^i_t &= H_t\left({X}^{\bar{\xi}_N}_t,Y^i_t,Z^i_t\right)dt -Z^i_t\sigma_t(X^{\bar{\xi}_N}_t) dW_t, \nonumber  \\ 
    Y^i_{T_{i}} &= Y^{i+1}_{T_{i}}+U_a\left(\xi_{i}\right). \nonumber
\end{align}
Finally, setting $Y_t := \sum_{i=1}^{N} Y^i_t\mathbbm{1}_{(T_{i-1},T_{i}]}$, and $Z_t:=  \sum_{i=1}^{N} Z^i_t\mathbbm{1}_{(T_{i-1},T_{i}]} $, we obtain that $(Y,Z)$ belongs to $\mathbb{D}_{\text{exp}}([0,T])\times \mathbb{H}([0,T])$ and solves \eqref{lemma.SDE}. Moreover, $Y_0\in \RR$ by the Blumenthal's zero–one law. \qed
\\
\\
\textit{Proof of Proposition 2:}
\\
\\
Firstly, consider a payment scheme $\bar{\xi}_{N}\in \Sigma_N$. By Proposition  \ref{BSDE.value.agent}, there exists a unique pair of process $(Y,Z) \in \mathbb{D}_{\text{exp}}([0,T])\times \mathbb{H}([0,T])$ satisfying 
\begin{align*}
        dY_t &= H_t\left({X}^{\bar{\xi}_{N-1}}_t,Y_t,Z_t\right)dt -Z_t\sigma_t(X^{\bar{\xi}_{N-1}}_t) dW_t,\quad \WW-a.s.\nonumber \quad T_{i-1} < t \leq T_i, \quad 1 \in \left\{ 1,\ldots, N\right\} \\ 
    Y_{T_{i}^{-}} &= Y_{T_{i}}+U_a\left(\xi_{i}\right), \quad \WW-a.s., \quad i \in \left\{1,\ldots,N \right\}, \\
    Y_{T} &= U_a(\xi_N).
\end{align*}
Let $0\leq t \leq T$, and $\alpha \in \mathcal{A}(\bar{\xi}_{N-1})$. Applying Ito Lemma and using Assumptions \ref{ass1}, \ref{ass2} and \ref{ass3}, we obtain
\begin{align*}
    J_a(t,\alpha,\bar{\xi}_N) &= \EE^{\PP^\alpha}\left[\mathcal{K}_{t,T}^\alpha Y_T +\sum_{i=1}^{N-1}\mathcal{K}^\alpha_{t,T_i}U_a\left(\xi_i\right)\mathbbm{1}_{t< T_i} -\int_{t}^{T}\mathcal{K}^\alpha_{t,s}c_s(X_s,\alpha_s)ds \bigg|\mathcal{F}_t\right] \\
    &=\EE^{\PP^\alpha}\left[\mathcal{K}_{t,T_{N-1}}^\alpha Y_{T_{N-1}}+\sum_{i=1}^{N-2}\mathcal{K}^\alpha_{t,T_i}U_a\left(\xi_i\right)\mathbbm{1}_{t< T_i} -\int_{t}^{T_{N-1}}\mathcal{K}^\alpha_{t,s}c_s(X_s,\alpha_s)ds\bigg|\mathcal{F}_t\right] \\
    &-\EE^{\PP^\alpha}\left[\int_{T_{N-1}}^T\mathcal{K}^\alpha_{T_{N-1},s}\left(H_s(X_s^{\bar{\xi}_{N-1}},Y_s,Z_s)-h_s(X_s^{\bar{\xi}_{N-1}},Y_s,Z_s,\alpha_s)\right) ds \bigg|\mathcal{F}_t\right].
\end{align*}
Repeating the same argument recursively, we obtain: 
\begin{align*}
    J_a(t,\alpha,\bar{\xi}_N) = Y_t -\EE^{\PP^\alpha}\left[\int_{t}^T\mathcal{K}^\alpha_{t,s}\left(H_s(X_s^{\bar{\xi}_{N-1}},Y_s,Z_s)-h_s(X_s^{\bar{\xi}_{N-1}},Y_s,Z_s,\alpha_s)\right) ds \bigg|\mathcal{F}_t\right].
\end{align*}
Therefore, 
\begin{equation*}
    V_t^a(\bar{\xi}_N)  \leq Y_t,\quad \WW-a.s.
\end{equation*}
and the previous upper bound is attained if and only if $\alpha^*_t = \hat{\alpha}(t,X^{\bar{\xi}_{N-1}}_t,Y_t,Z_t)$ $dt\otimes \WW-a.e.$, where $\hat{\alpha}(t,x,y,z) \in \argmax h_t(x,y,z,\cdot)$. Note that $\alpha^*$ is an $\FF$-adapted process. Indeed, using assumptions \ref{ass1}, \ref{ass2}, \ref{ass3}, we obtain that $\hat{\alpha}$ satisfies
\begin{align*}
      \frac{\partial }{\partial a} \lambda(t,x,\hat{\alpha}(t,x,y,z))z- \frac{\partial }{\partial a}k_a(x,\hat{\alpha}(t,x,y,z))y = \frac{\partial }{\partial a} c(t,x,\hat{\alpha}(t,x,y,z)).
\end{align*}
Hence, applying the measurable selection theorem (Theorem 7.49, \cite{stochcontroldiscrete}), we obtain that the process $\alpha^*$ is $\FF$-adapted. Moreover, using assumptions \ref{ass1}, \ref{ass2} and \ref{ass3}  we obtain that there exits $C>0$ such that
$|\hat\alpha(t,x,y,z)| \leq C(1+|z|+|y|)$. Using the later estimate and the hypothesis on $Z$, we obtain that $\alpha^* \in \mathcal{A}(\bar{\xi}_N)$. 
\qed

\subsection{Proof of Theorem \ref{hjb:pp:in}}
Before proving Theorem \ref{hjb:pp:in} we introduce the following auxiliary result. For $n\geq 1$, we denote $\RR^{+.n}$ the set of $n$-tuples with positive entries. 
\begin{definition}
    Let $\delta:=(b,c,M) \in \RR^{+,3}$,  $a \in \RR^+$, $\gamma > 1$. We introduce the function $\phi^{\gamma,a,\delta} : [0,T]\times [0,\infty)\mapsto \RR$ defined by
    \begin{align*}
        &\phi^{\gamma,a,\delta}(t,y) :=  -ay^{\gamma} + be^{\frac{(T-t)}{T}} y^{\frac{1}{M}} +e^{c(T-t)} (1-e^{-y}). 
    \end{align*}
\end{definition}
\begin{lemma}\label{lemma_supersol}
        For all $\gamma\geq 1$, there exits $\delta_0\in \RR^{+,3}$ such that
        \begin{equation*}
            \min_{a\in\left\{1,\frac{1}{2^{\gamma - 1}},\ldots,\frac{1}{N^{\gamma - 1}} \right\}}\left\{-\phi_t^{\gamma,a,\delta_0}(t,y)- \sup_{|z|\leq K}\left\{z+ \frac{1}{2}\left(\phi^{\gamma,a,\delta_0}_{yy}(t,y)+\phi^{\gamma,a,\delta_0}_y(t,y)\right)z^2 \right\}-k_ay\phi^{\gamma,a,\delta_0}_y(t,y)\right\} \geq 0,
        \end{equation*}  
        for all $(t,y) \in [0,T]\times(0,\infty) $.
\end{lemma}
\begin{proof}
    
    Firstly,we consider the case $\gamma > 1$. In this case, for any $M > \frac{1}{\gamma-1}$, we obtain
    \begin{align*}
        &\max_{a\in \left\{1,\ldots,{\frac{1}{N}}^{\gamma - 1}\right\}}\sup_{(t,y) \in [0,T]\times(0,\infty)}\left\{
       \phi_{yy}+\phi_y \right\}\\
       =& b e^{\frac{(T-t)}{T}}\frac{1}{M}\left(  y^{\frac{1}{M} - 1} -  \frac{M -1}{M} y^{\frac{1}{M} - 2}\right)  - a\gamma y^{\gamma - 1}  - a\gamma (\gamma - 1)y^{\gamma - 2}\\
     =& y^{\gamma - 2} \left( b e^{\frac{(T-t)}{T}}\frac{1}{M}\left(  y^{\frac{1}{M} + 1 - \gamma} -  \frac{M -1}{M} y^{\frac{1}{M} - \gamma}\right) - a\gamma y  - a\gamma(\gamma - 1)  \right)\\
     \leq& y^{\gamma - 2} \left( b e^{\frac{(T-t)}{T}}\frac{1}{M}\left(  y^{\frac{1}{M} + 1 - \gamma} -  \frac{M -1}{M} y^{\frac{1}{M} - \gamma}\right) - \frac{\gamma(\gamma - 1)}{N^{\gamma - 1}}  \right)\\
     <& y^{\gamma - 2} \left( b e\frac{1}{M}  {\left(\frac{\gamma M - 1}{M}\right)}^{\frac{1}{M} - \gamma}  \left(\frac{M - 1}{M(\gamma - 1) - 1}\right)^{\frac{1}{M} - \gamma + 1 }   - \frac{\gamma(\gamma - 1)}{N^{\gamma - 1}}   \right)\\
     <& y^{\gamma - 2} \left( b e\frac{1}{M}    - \frac{\gamma(\gamma - 1)}{N^{\gamma - 1}}   \right).
    \end{align*}
    The third inequality comes from applying the first order condition  to the function $  g_1(y):=y^{\frac{1}{M} + 1 - \gamma} -  \frac{M -1}{M} y^{\frac{1}{M} - \gamma}$, which is satisfied at {$y^{*} := \frac{(M-1)(\gamma M - 1)}{M(M(\gamma - 1) - 1)} \geq  \frac{M - 1}{M}$}. The fourth inequality comes from the monotonicity of the function
    $$g_2(M) :=  {\left(\frac{\gamma M - 1}{M}\right)}^{\frac{1}{M} - \gamma} \left(\frac{M - 1}{M(\gamma - 1) - 1}\right)^{\frac{1}{M} - \gamma + 1 },$$
 restricted to $\left(\frac{1}{\gamma-1},\infty\right)$.
    
    Then, for any $ 0< b < \frac{1}{e} M (\frac{1}{N})^{\gamma - 1} \gamma(\gamma - 1)$, the following inequality holds
    \begin{equation}\label{vyvyynegative}
        \max_{a\in \left\{1,\ldots,{\frac{1}{N}}^{\gamma - 1}\right\}}\sup_{(t,y) \in [0,T]\times(0,\infty)}\left\{
       \phi_{yy}+\phi_y \right\} < 0.
    \end{equation}
   Using (\ref{vyvyynegative}), we have that for all $ a\in \left\{1,\ldots,{\frac{1}{N}}^{\gamma - 1}\right\}$, and $\delta= (b,c,M)\in \RR^{+,3}$, $M > \frac{1}{\gamma - 1} $, $ b<\frac{1}{e} M (\frac{1}{N})^{\gamma - 1} \gamma(\gamma - 1)$:
\begin{align*}\label{supervineq}
         -\sup_{|z|\leq K}\left\{z+ \left(\phi^{a,\delta}_{yy}(t,y)+\phi^{a,\delta}_y(t,y)\right){\frac{z^2}{2}} \right\} 
         &\geq -\sup_{z\in \RR}\left\{z+ \left(\phi^{a,\delta}_{yy}(t,y)+\phi^{a,\delta}_y(t,y)\right)  {\frac{z^2}{2}}  \right\} \\
          &=\frac{1}{2\left( be^{(T-t)p}\left(\frac{1}{M} y^{\frac{1}{M} - 1}  - \frac{M - 1}{M^{2}}y^{\frac{1}{M}-2}\right) - \gamma a y^{\gamma - 1}   - \gamma(\gamma - 1)a y^{\gamma - 2} \right)}
\end{align*}

Hence, for any $\delta = (b,c,M)\in \RR^{+,3}, \quad M > \frac{1}{\gamma - 1} ,\quad b<\frac{1}{e} M (\frac{1}{N})^{\gamma - 1} \gamma(\gamma - 1),$ and  $a\in \left\{1,\ldots,{\frac{1}{N^{\gamma - 1}}}\right\}$,
 the following inequality holds
\begin{align}
        &-\phi_t^{a,\delta_0}(t,y)- \sup_{|z|\leq K}\left\{z+ \left(\phi^{a,\delta_0}_{yy}(t,y)+\phi^{a,\delta_0}_y(t,y)\right){\frac{z^2}{2}} \right\}-k_ay\phi^{a,\delta_0}_y(t,y)  \nonumber \\
        &\geq \frac{1}{2\left( be^{\frac{(T- t)}{T}}(\frac{1}{M} y^{\frac{1}{M} - 1}  - \frac{M - 1}{M^{2}}y^{\frac{1}{M}-2}) - \gamma a y^{\gamma - 1}   - \gamma(\gamma - 1)a y^{\gamma - 2} \right)} +b (\frac{1}{T} - \frac{k_a}{M}) e^{\frac{(T -t)}{T}} y^{\frac{1}{M}} \nonumber \\
        &+ \gamma(k_a) a y^{\gamma} + e^{(T-t)c}(c - ce^{-y} - k_ay e^{-y}),
\end{align}

    We will show that for all  $$\delta_0:=(b,M,c)\in \RR^{+,3},  M > \max\{k_aT, \frac{1}{\gamma - 1},2\}, c > 3k_a, b<\frac{1}{e} M (\frac{1}{N})^{\gamma - 1} \gamma(\gamma - 1),$$ the inequality
    \begin{equation}\label{eq.stronger}
     1 \leq 2c\left(1 - e^{-y} - \frac{k_a}{c}ye^{-y}\right)\left( -  be^{\frac{(T- t)}{T}}\left(\frac{1}{M} y^{\frac{1}{M} - 1}  - \frac{M - 1}{M^{2}}y^{\frac{1}{M}-2}\right) + \gamma a y^{\gamma - 1}   + \gamma(\gamma - 1)a y^{\gamma - 2} \right),
    \end{equation}
    holds for any $a \in \{1,\ldots,{\frac{1}{N^{\gamma - 1}}}\}$. The previous implies that 
     \begin{align*}
          &\min_{a\in \{1,\ldots,\frac{1}{N}\}}\left\{-\phi_t^{a,\delta_0}(t,y)- \sup_{|z|\leq K}\left\{z+ \frac{1}{2}\left(\phi^{a,\delta_0}_{yy}(t,y)+\phi^{a,\delta_0}_y(t,y)\right)z^2 \right\}-k_ay\phi^{a,\delta_0}_y(t,y)\right\} \geq 0.
     \end{align*}
    
   To show \eqref{eq.stronger}, we use the following estimate:
    \begin{equation}\label{estimate.1st}
    1 - e^{-y} - \frac{k_a}{c}y e^{-y} \geq \frac{1}{4}y\mathds{1}_{\{0 \leq y < \frac{1}{2}\}} + \frac{1}{4}\mathds{1}_{\{y\geq\frac{1}{2}\}}.
    \end{equation}
    From \eqref{estimate.1st}, we obtain
    \begin{align*}
    &2c\left(1 - e^{-y} - \frac{k_a}{c}ye^{-y} \right)\left( -b e^{\frac{(T-t)}{T}}\left(\frac{1}{M}y^{\frac{1}{M} - 1} - \frac{M-1}{M^{2}} y^{\frac{1}{M} - 2}\right) +  \gamma a y^{\gamma - 1}   + \gamma(\gamma - 1)a y^{\gamma - 2} \right)\\
    \geq&\frac{c}{2}\left( -b e^{\frac{(T-t)}{T}}\left(\frac{1}{M}y^{\frac{1}{M} - 1} - \frac{M-1}{M^{2}} y^{\frac{1}{M} - 2}\right) + \gamma a y^{\gamma - 1}     + \gamma(\gamma - 1)a y^{\gamma - 2} \right)\\ 
    \geq& \frac{c y^{\gamma - 1}}{2}\left( -b e^{\frac{(T-t)}{T}}\left(\frac{1}{M}y^{\frac{1}{M} - \gamma} - \frac{M-1}{M^{2}} y^{\frac{1}{M} - 1 - \gamma}\right) + \gamma a  \right)\\ 
    \geq&   {\frac{ca}{2^{\gamma}}} \left( -\frac{be}{a}{\left(\frac{M - 1}{\gamma M - 1}\right)}^{{\frac{1}{M}} - \gamma}  \left(\frac{(\gamma+ 1)M - 1}{M}\right)^{\frac{1}{M} - \gamma - 1} + \gamma   \right) \\
    \geq&  {\frac{ca}{2^{\gamma}}} \left( -\frac{be}{a}{\left(\frac{M^{*} - 1}{\gamma M^{*} - 1}\right)}^{{\frac{1}{M^{*}}} - \gamma}  + \gamma   \right),
    \end{align*}
    where $M^* := \max \{2,\frac{1}{\gamma - 1}\} $.
    
    The first inequality comes from \eqref{estimate.1st}, and the third inequality is obtained from applying the first order condition to $(\frac{1}{M}y^{\frac{1}{M} - 1} - \frac{M-1}{M^{2}} y^{\frac{1}{M} - 2})$. The second last inequality is a consequence of the nonnegative and nonincreasing property of the function  
    $$g_3(M):={\left(\frac{M - 1}{\gamma M - 1}\right)}^{{\frac{1}{M}} - \gamma}\left(\frac{(\gamma+ 1)M - 1}{M}\right)^{\frac{1}{M} - \gamma - 1}$$ 
    in the domain $M \geq M^*$. 
    Hence, for all $\delta_0^1 := (b,c,M) \in \RR^{3,+}$ satisfying  
     \begin{equation}\label{delta.1}
      c >\frac{ 2^\gamma  }{ a \left( -\frac{b}{a}e{(\frac{M^{*} - 1}{\gamma M^{*} - 1})}^{{\frac{1}{M^{*}}} - \gamma}  + \gamma   \right) } , \quad
      M \geq \max \left\{2,\frac{1}{\gamma - 1}\right\} , \quad
      b< \frac{\gamma a }{e \left( -\frac{b}{a}e{(\frac{M^{*} - 1}{\gamma M^{*} - 1})}^{{\frac{1}{M^{*}}} - \gamma}  + \gamma   \right)},
     \end{equation}
    we obtain that \eqref{eq.stronger} holds for all $y \geq \frac{1}{2}$.
     
    Next, we consider $0 < y < \frac{1}{2}$. In this case, applying the estimate \eqref{estimate.1st}, we obtain that for all $a \in \{1,\ldots,\frac{1}{N^{\gamma-1}} \}$:
     \begin{align*}
    &2c\left(1 - e^{-y} -\frac{k_a}{c}y e^{-y} \right)\left( -b e^{\frac{(T-t)}{T}}\left(\frac{1}{M}y^{\frac{1}{M} - 1} - \frac{M-1}{M^{2}} y^{\frac{1}{M} - 2}\right) +  \gamma a y^{\gamma - 1}   + \gamma(\gamma - 1)a y^{\gamma - 2}  \right)\\
    \geq&\frac{c}{2}y\left( b e^{\frac{(T-t)}{T}}\left( - \frac{1}{M}y^{\frac{1}{M} - 1} + \frac{M-1}{M^{2}} y^{\frac{1}{M} - 2}\right) +  \gamma a y^{\gamma - 1}   + \gamma(\gamma - 1)a y^{\gamma - 2} \right)\\
    \geq&\frac{c}{2}y^{\gamma - 1}\left( b e^{\frac{(T-t)}{T}}\left( - \frac{1}{M}y^{\frac{1}{M} + 1 - \gamma} + \frac{M-1}{M^{2}} y^{\frac{1}{M} - \gamma}\right)    + \gamma(\gamma - 1)a  \right) \\
    \geq& \frac{c}{2}y^{\gamma - 1}\left( b e^{\frac{(T-t)}{T}}\left( - \frac{1}{M}y^{\frac{1}{M} + 1 - \gamma} + \frac{M-1}{M^{2}} y^{\frac{1}{M} - \gamma}\right)    + \frac{\gamma(\gamma - 1)}{N^{\gamma-1}}  \right). 
    \end{align*}
    Next, we define the following functions
    \begin{equation}
        \begin{split}
             \underline{\varphi}(t,y) &:=   \frac{1}{2}y^{\gamma - 1}\left(b e^{(T-t)p}\left( - \frac{1}{M}y^{\frac{1}{M} + 1 - \gamma} + \frac{M-1}{M^{2}} y^{\frac{1}{M} - \gamma} \right) + \frac{\gamma(\gamma - 1)}{N^{\gamma-1}} \right),\\
                      \underline{\varphi}^{1}(t,y) &:=  \left(b e^{(T-t)p}\left( - \frac{1}{M}y^{\frac{1}{M} + 1 - \gamma} + \frac{M-1}{M^{2}} y^{\frac{1}{M} - \gamma} \right) + \frac{\gamma(\gamma - 1)}{N^{\gamma-1}} \right) .
        \end{split}
    \end{equation}
    For all $0 \leq t\leq T$, and $\bar{\varphi}>0$, there exists $y_1:= y_1(\bar{\varphi},b,M)$, such that $\underline{\varphi}^1(t,y) > \bar{\varphi}$,  $ y \in (0,y_1]$. 
    \\
    Additionally, we know that $ - \frac{1}{M}y^{\frac{1}{M} + 1 - \gamma} + \frac{M-1}{M^{2}} y^{\frac{1}{M} - \gamma} > 0$, for all $ 0< y_1< y < \frac{M-1}{M} $.  Applying the first order condition to $\underline{\varphi}^{1}(t,\cdot)$, we obtain that $y^* =\argmin \underline{\varphi}(t,\cdot) = \frac{(M-1)(\gamma M - 1)}{M(M(\gamma - 1) - 1)} \geq \frac{M - 1}{M}  > \frac{1}{2}$. The previous implies that for all $0 \leq t \leq T$, the function $\underline{\varphi}^{1}(t,\cdot)$ is decreasing on the domain  $ 0< y < \frac{(M-1)(\gamma M - 1)}{M(M(\gamma - 1) - 1)} $. 
    Hence,
    $$\underline{\varphi}^{1}(t,y) \geq \underline{\varphi}^{1}\left(t,\frac{1}{2}\right) >  \underline{\varphi}^{1}\left(t,\frac{M - 1}{M}
    \right) \geq \gamma(\gamma - 1) > 0 ,\quad y \in \left[y_1,\frac{1}{2}\right].$$
   Additionally, using that $\underline{\varphi}(t,\cdot)$ is continuous on $[y_1,\frac{1}{2}]$, we obtain that for all $c > \frac{2}{\gamma(\gamma - 1)} $,   $$\min_{y\in [y_1,\frac{1}{2}]} \underline{\varphi}(t,y) \geq \varphi^* > 0 ,$$
   where $\varphi^*(\bar{\varphi},b,M) := \left(y_1(\bar{\varphi},b,M)\right)^{\gamma-1} $.
    Hence, for all $\delta_0^2 := (b,c,M)\in \RR^{+,3} $ satisfying  
    \begin{equation}\label{delta.2}
        c >N^{\gamma-1}\min\left\{\bar{\varphi},\varphi^*(\bar{\varphi},b,M)\right\} ,\quad M \geq \frac{1}{\gamma - 1} ,\quad b< \min \left\{\frac{1}{e} \gamma(\gamma - 1) M \left(\frac{1}{N}\right)^{\gamma - 1} ,\frac{1}{N^{\gamma-1}}\right\},
    \end{equation}
    we obtain that \eqref{eq.stronger} holds for all $y \in \left(0,\frac{1}{2}\right]$. \\
    
    Combining (\ref{delta.1}) and (\ref{delta.2}), we conclude that there exists $\delta_0 \in \RR^{3,+}$ for which the inequality in Lemma \ref{lemma_supersol} holds.
    \qed

\end{proof}

\begin{definition}
    Let $\gamma > 1$ and $\delta_0 \in \RR^{+,3}$  be the 3-tuple determined in Lemma \ref{lemma_supersol}. We introduce the function ${\phi}^{\gamma,\delta_0}: [0,T]\times [0,\infty) \mapsto \RR $ defined as
    \begin{align*}
        \phi^{\gamma,\delta_0}(t,y) := \sum_{i=1}^N \phi^{\gamma,\frac{1}{(1+N-i)^\gamma},\delta_0}(t,y)\mathbbm{1}_{(T_{i-1},T_{i}]}. 
    \end{align*}
\end{definition}
\begin{remark}
The function $\phi^{\gamma,\delta_0}$ plays an important role in our approach. It is carefully designed to obtain an appropriate upper bound of the principal´s value function. The later is crucial to successfully employ the results from \cite{bouchard2012weak}.
\end{remark}
Next, we introduce the open state constraint value function:
\begin{equation*}
    \hat{V}(t,x,y) := \sup_{Y_0 \geq R_a}\sup_{(Z,\bar{\xi}_{N-1}) \in \hat{\mathcal{U}}(t,x,y)}\EE\left[X^{t,x,Z,\bar{\xi}_{N-1}}_T-\left(Y^{t,y,Z,\bar{\xi}_{N-1}}_T\right)^{\gamma}  \right],
\end{equation*}
where 
\begin{align*}
     &X^{t,x,Z,\bar{\xi}_{N-1}}_s = x-\sum_{j=1}^{N-1} \xi_{j}\mathbbm{1}_{t<T_j\leq s} +\int_t^sZ_rdr +  B_t-B_s, \\
     &Y^{t,x,Z,\bar{\xi}_{N-1}}_s = y-\sum_{j=1}^{N-1} U_a\left(\xi_{j}\right)\mathbbm{1}_{t<T_j\leq s} +\int_t^s\left(\frac{1}{2}Z^2_r+k_aY_r\right)dr + \int_t^sZ_rdB_r, 
\end{align*}
and $\hat{\mathcal{U}}(t,x,y):= \left\{(Z,\bar{\xi}_{N-1}) \in \mathcal{U}(t,T) \big| \hspace{1 mm} Y_t^{t,y,Z,\bar{\xi}_N}>0 \hspace{1 mm} dt\otimes \PP -a.e. \right\}$.
\\
We introduce the lower and upper semicontinuous envelopes of $\hat{V}$, defined for all $(t,x,y) \in [0,T]\times \RR\times (0,\infty)$: 
\begin{align*}
    \hat{V}_*(t,x,y)&:=\liminf_{(t',x',y')\rightarrow (t,x,y), y' > 0}\hat{V}(t',x',y'), \\
    \hat{V}^*(t,x,y)&:=\limsup_{(t',x',y')\rightarrow (t,x,y),  y' > 0}\hat{V}(t',x',y').
\end{align*}
        Firstly, we show Theorem $2.1, 2.2, 2.3,$ simultaneously by backward induction on the regions  $\mathcal{R}_i := [T_{i-1},T_{i})\times \RR\times [0,\infty) $, $1\leq i \leq N$.  \\

    \textit{Proof of Theorem 2.1 on $\mathcal{R}_N$:}\\

Firstly, consider $(t,x,y) \in (T_{N-1},T]\times \RR \times (0,\infty)$, a control process $(Z,\bar{\xi}_{N-1}) \in \hat{\mathcal{U}}(t,x,y)$, and an $\FF$-stopping time $\tau$. Applying Ito Lemma:
\begin{align}\label{super.sol.ineq}
  x+\phi^{\gamma,\delta_0}(t,y) &= X^{t,x,\bar{\xi}_N,Z}_{\tau}+\phi^{\gamma,\delta_0}(\tau,Y^{t,x,\bar{\xi}_N,Z}_\tau) \\
  &- \int_t^\tau \left(\phi_t^{\gamma,\delta_0}(s,Y^{t,x,\bar{\xi}_N,Z}_s)+\mathcal{L}^{Z_s}\phi^{\gamma,\delta_0}(s,Y^{t,x,\bar{\xi}_N,Z}_s) \right)ds+M^{Z}_t \nonumber
\end{align}
where 
\begin{equation*}
    M_\tau^{t,Z} :=W_{\tau}-W_t+ \int_t^\tau \phi_y(s,Y_s^{t,y,\bar{\xi}_N,Z})dW_s,
\end{equation*}
and, for any $z \in \RR$, the operator $\mathcal{L}^z$ is defined as
\begin{align*}
\mathcal{L}^z f (y)&: =  z+ \frac{1}{2}\left(f_{yy}(y)+f_y(y)\right)z^2+k_a f_y(y)y.
\end{align*}
\\
Note that $M^{t,Z}$ is a local martingale. 
Next, we introduce a localizing sequence $(\tau_n)_{n\in \NN} $ defined by 
$$\tau_n := \left\{s\geq t : \bigg|\phi_y\left(s,Y_s^{t,y,\bar{\xi}_N,Z}\right)\bigg| \geq n  \right\}\wedge T.$$ 
\\
Fixing $\tau := \tau_n$, and taking expectations in \eqref{super.sol.ineq}, we obtain from Lemma \ref{lemma_supersol}: 
\begin{equation*}
    x+\phi^{\gamma,\delta_0}(t,y) \geq \EE\left[ X^{t,x,\bar{\xi}_N,Z}_{\tau_n}+\phi^{\gamma,\delta_0}(\tau_n,Y^{t,y,\bar{\xi}_N,Z}_{\tau_n})\right].
\end{equation*}
Next, using that the controls admissible $Z$ are uniformly bounded and standard SDE estimates, we have
\begin{align*}
\EE\left[ \sup_{T_{N-1} \leq s\leq T} \left|X^{t,x,\bar{\xi}_N,Z}_{s}+\phi^{\gamma,\delta_0}(s,Y^{t,y,\bar{\xi}_N,Z}_{s})\right|^2\right] < \infty. 
\end{align*}
Thus, $\left(X^{t,x,\bar{\xi}_N,Z}_{\tau_n}+\phi^{\gamma,\delta_0}(\tau_n,Y^{t,y,\bar{\xi}_N,Z}_{\tau_n})\right)_{n\geq 1}$ is uniformly integrable. Hence, using that $\lim_{n\rightarrow\infty} \tau_n = T, \PP-a.s.$, and the continuity of $\phi^{\gamma,\delta_0}$:
\begin{equation}\label{ineq.value}
    x + \phi^{\gamma,\delta_0}(t,y)  \geq \EE\left[ X^{t,x,\bar{\xi}_N,Z}_{T}+\phi^{\gamma,\delta_0}(T,Y^{t,y,\bar{\xi}_N,Z}_{T})\right].
\end{equation}
Taking the supremum over $(Z,\bar{\xi}_{N-1}) \in \hat{\mathcal{U}}(t,x,y)$ in \eqref{ineq.value} and using that $\phi(T,y) \geq -y^{\gamma}$ for all $y>0$, we obtain
\begin{align}\label{limit.lower.envelope}
     x + \phi^{\gamma,\delta_0}(t,y) \geq \hat{V}(t,x,y).
\end{align}
Moreover, taking the control $\hat{Z}:=0$, we have 

\begin{equation*}
    x-e^{\gamma k_a(T-t)}y^\gamma \leq \hat{V}(t,x,y) \leq x + \phi^{\gamma,\delta_0}(t,y).
\end{equation*}
The previous implies
\begin{equation}\label{eq.cont.value}
    x = \hat{V}_*(t,x,0) = \lim_{(t',x',y') \rightarrow (t,x,0), y'>0} \hat{V}_*(t,x,y).
\end{equation}
Next we apply (\cite{bouchard2012weak}, Proposition $4.11$). Firstly, we check that all the assumptions of the Proposition 4.11 are satisfied. Indeed, equation \eqref{eq.cont.value} implies that $\hat{V}_{*}$ is continuous on $\{y=0\}$. Due to the uniformly boundedness of the controls we have that the drift of the state processes $(X,Y)$ grows linearly in $(x,y) \in \RR\times (0,\infty)$. Additionally, for all $y>0$, the control $\hat{Z} := 0$ satisfies $Y_s^{t,y,\hat{Z}}>0, \PP-a.s.$
Hence, $\hat{V}$ is the unique (discontinuous) viscosity solution to the following state constrained HJB equation: 
\begin{align}\label{pde.last.period}
    -\varphi_t - \sup_{|z|\leq K}\left\{\frac{1}{2}\left(\varphi_y+\varphi_{yy}\right)z^2+\varphi_x z +\frac{1}{2}\varphi_{xx}z^2+2\varphi_{xy}z\right\} -   k_a y \phi_y   &= 0,\\
    \varphi(T,x,y) &= x -y^{\gamma}, \nonumber
\end{align}
in the class of functions with polynomial growth and  lower semi continuous envelope continuous at $[T_{N-1},T]\times\{y=0\}$.  Moreover, using (\cite{bouchard2012weak}, Corollary 4.13), we have that $V(t,x,y) = \hat{V}(t,x,y)$, for all $(t,x,y) \in (T_{N-1},T]\times \RR\times (0,\infty)$. Using that $V$ is continuous on $[T_{N-1},T]\times \RR \times [0,\infty)$ and the fact that the value is separable in $(x,y)$\footnote{We show this property trivially from the definition of the principal's value in \eqref{principal.value.benchmark}. }, there exists a continuous function $v : (T_{N-1},T]\times [0,\infty) \mapsto \RR$, such that
\begin{equation*}
    V(t,x,y) = x + v(t,y).
\end{equation*}
Plugging in the last expression into \eqref{pde.last.period}, we obtain that $v$ is the unique viscosity solution to the following state constrained HJB equation: 
\begin{align*}
        -\varphi_t - \sup_{|z|\leq K}\left\{\frac{1}{2}\left(\varphi_y+\varphi_{yy}\right)z^2+z\right\}   -  k_a y v_y    &= 0, \quad (t,y) \in [T_{N-1},T]\times (0,\infty),\\
    \varphi(T,x,y) &= -y^{\gamma},\quad y>0,
\end{align*}
in the class of functions with polynomial growth and  lower semi continuous envelope continuous at $\{y=0\}$. \qed \\

\textit{Proof of Theorem $2.2$ on $\mathcal{R}_N$:}\\

From the proof of $3.1$ on $\mathcal{R}_N$, we know that $V(T_{N-1},x,y) = x + v(T_{N-1},y)$ is continuous on $\RR\times [0,\infty)$. We write 
            $$
            f_i(x, y)=\max_{\eta \in \beta(x, y)} h(x, y, \eta)
            $$

            where $h(x, y, \eta)=V\left(T_{N-1}, x- \eta^{\gamma},  
         y -\eta \right)$, and $\beta(x, y)=\left\{\eta \in[0,\infty): \eta \leq y \right\}$.\\Clearly, $h : \mathbb{R} \times [0,\infty) \times [0,y] \rightarrow \mathbb{R}$ is a continuous function due to the continuity of  $V(T_{N-1},\cdot,\cdot)$ shown in the previous section. Moreover, $\beta(x,y)$ is a continuous set-valued map with non-empty compact values by Theorems 17.20, and 17.21 in \cite{guide2006infinite}.  Using Berge's maximum theorem and (\cite{guide2006infinite},Lemma 17.30) we obtain that $f_i$ is continuous and the lowest maximizer $\eta_{i}^{*}$ is lower semicontinuous. \qed \\
         \\
         \textit{Proof of Theorem $2.3$ on $\mathcal{R}_N$:}
         \\
         \\
        Let $t \in [T_{N-2},T_{N-1})$. We will show that for all $\epsilon>0$, and $(t,x,y) \in [T_{N-2},T_{N-1})\times \RR \times [0,\infty)$: 
        \begin{equation}\label{eq::cdppn}
            V(t, x, y)=\sup _{(z, \xi_{N-1}) \in \mathcal{U}(t,x, y)} \mathbb{E}\left[V\left(T_{N-1}+\varepsilon, X_{T_{N-1}+\varepsilon}^{t, x, Z, \xi_{N-1}}, Y_{T_{N-1}+\varepsilon}^{t, y,Z,\xi_{N-1}}\right)\right].
        \end{equation}
        Note that the right hand side of \eqref{eq::cdppn} is well defined as we showed that $V$ is continuous in $(T_{N-1},T_{N-2}) \times \RR \times [0,\infty)$ and therefore, measurable. 
        Next, we consider the principal's objective:
        $$J(t,x,y,Z,\xi_{N-1}) = \mathbb{E}\left[X^{t,x,Z,\xi_{N-1}}_{T} - \left(Y^{t,y,Z,{\xi}_{N-1}}_{T}\right)^{\gamma}\right].$$
        Evidently, we have that, for all $ (t,x,y) \in (T_{N-2},T_{N-1}) \times \mathbb{R} \times [0,\infty)$, the following inequality holds:
        \begin{equation}\label{ineq::concate1}
            J(t,x,y,Z,\xi_{N-1}) \leq \sup _{(Z, \bar{\xi}_{N-1}) \in \mathcal{U}(t,x, y)} \mathbb{E}\left[V\left(T_{N-1}+\varepsilon, X_{T_k+\varepsilon}^{t, x, Z, \xi_{N-1}}, Y_{T_k+\varepsilon}^{t, y,Z,\xi_{N-1}}\right)\right].
        \end{equation}
        Hence, 
        \begin{equation*}
            V(t,x,y) \leq \sup_{(Z,\xi_{N-1})\in {\mathcal{U}}(t,x,y)}\mathbb{E}\left[V\left(T_{N-1}+\varepsilon, X_{T_k+\varepsilon}^{t, x, Z, \xi_{N-1}}, Y_{T_k+\varepsilon}^{t, y,Z,\xi_{N-1}}\right)\right].
        \end{equation*}
        Repeating the same argument using $\hat{\mathcal{U}}$ as the space of admissible controls, we obtain 
        \begin{align*}
             \hat{V}(t,x,y) \leq \sup_{(Z,\xi_{N-1})\in \hat{\mathcal{U}}(t,x,y)}\mathbb{E}\left[V\left(T_{N-1}+\varepsilon, X_{T_k+\varepsilon}^{t, x, Z, \xi_{N-1}}, Y_{T_k+\varepsilon}^{t, y,Z,\xi_{N-1}}\right)\right].
        \end{align*}
        Next, we show the reverse inequality in \eqref{eq::cdppn}. Firstly, we observe that
        \begin{equation}
            V(t,x,y) \geq \hat{V}(t,x,y) \geq \hat{V}(T_{N-1}+\epsilon,x,e^{k_a\max\{T_{N-1}+\epsilon-t,0\}}y), 
        \end{equation}
        where the second inequality holds by considering the controls $\hat{Z}_t := Z_t\mathbbm{1}_{t > T_{N-1}+\epsilon}$, where $Z \in \mathcal{U}(t,x,e^{ka(T_{N-1}-t)}y)$. 
        \\
        Additionally, we observe that the mapping $(t,x,y) \mapsto \hat{V}(T_{N-1}+\epsilon,x,  e^{k_a\max\{ T_{N-1}+\epsilon-t,0\} }y)$ is continuous on $ \mathcal{R}_N$ as $V$ is continuous in $\mathcal{R}_N$.  Hence, for some $\delta>0$ small enough, we consider the set $B := (t-\delta,T_{N-1}+\epsilon)\times \RR \times (0,\infty) $ and realize that $T_{N-1}+\epsilon$ is the first exist time of $\left(s,X^{t,Z,\xi_{N-1}}_s,Y^{t,Z,\xi_{N-1}}_s\right)_{s\geq t}$ from $B$. 
       Using  (\cite{bouchard2012weak}, Lemma 4.9 (ii)) we have
       \begin{equation}\label{ineq.dpp2}
        V(t,x,y)\geq    \hat{V}(t,x,y) \geq \EE \left[\hat{V}(T_{N-1}+\epsilon,X^{t,x,Z,\xi_{N-1}}_{T_{N-1}+\epsilon},Y^{t,y,Z,\xi_{N-1}}_{T_{N-1}+\epsilon}) \right],
       \end{equation}
        for all $(Z,\xi_{N-1}) \in \mathcal{U}(t,x,y)$. 
        Finally, noticing that $\hat{V} = V$ on $\mathcal{R}_N$ and taking supremum on \eqref{ineq.dpp2}, we obtain:
        \begin{equation}
            V(t,x,y) \geq \sup_{(Z,\xi_{N-1})\in \hat{\mathcal{U}}(t,x,y)} \EE \left[V(T_{N-1}+\epsilon,X^{t,x,Z,\xi_{N-1}}_{T_{N-1}+\epsilon},Y^{t,y,Z,\xi_{N-1}}_{T_{N-1}+\epsilon}) \right].
        \end{equation}
         Due to the continuity and local boundedness of $V$ on $\mathcal{R}_N$, we have taking $\epsilon \to 0$:
        \begin{equation}\label{v.region}
        V(t, x, y) \leq \sup _{Z \in \mathcal{U}(t,x, y)} \mathbb{E}\left[f_{N-1}\left( X_{T^{-}_{N-1}}^{t, x,Z}, Y_{T^{-}_{N-1}}^{t, y,Z}\right)\right] \quad (t,x,y) \in (T_{N-2},T_{N-1}) \times \mathbb{R} \times [0,\infty),
        \end{equation}
        \begin{equation}\label{v.hat.region}
        \hat{V}(t, x, y) = \sup _{Z  \in \hat{\mathcal{U}}(t,x, y)} \mathbb{E}\left[f_{N-1}\left( X_{T^{-}_{N-1}}^{t, x,Z}, Y_{T^{-}_{N-1}}^{t, y,Z}\right)\right] \quad (t,x,y) \in (T_{N-2},T_{N-1}) \times \mathbb{R} \times (0,\infty).
        \end{equation}
        Let $(t,x,y) \in (T_{N-2},T_{N-1})\times [0,\infty]\times \RR $.
        We consider the sequence $(t_n,x_n,y_n) \to (T_{N-1},x,y)$, $t_n < T_{N-1}$, $y_n >0$. Based on the claim \eqref{v.hat.region}, we have the following inequalities :
        $$\mathbb{E}\left[f_{N-1}\left(X_{T^{-}_{N-1}}^{t_n, x_n, Z^n}, Y_{T^{-}_{N-1}}^{t_n, y_n, Z^n}\right)\right] \leq \hat{V}\left(t_n, x_n, y_n\right) \leq \mathbb{E}\left[f_{N-1}\left(X_{T^{-}_{N-1}}^{t_n, x_n, Z^n}, Y_{T^{-}_{N-1}}^{t_n, y_n, Z^n}\right)\right]+\frac{1}{n},$$
        where $Z^{n}$ is a $\frac{1}{n}$-optimal control.
        Moreover, we have
        \begin{equation*}
        \begin{split}
            & \mathbb{E}\left[\left|\left(X_{T^{-}_{N-1}}^{t_n, x_n, Z^n}, Y_{T^{-}_{N-1}}^{t_n, y_n,Z^n}\right)-(x, y)\right|^2\right] \\
            & \leq 2 \mathbb{E}\left[\left|\left(X_{T^{-}_{N-1}}^{t_n, x_n,  Z^n}, Y_{T^{-}_{N-1}}^{t_n, y_n,Z^n}\right)-\left(x_n, y_n\right)\right|^2\right]+2\left(x_n-x\right)^2+2\left(y_n-y\right)^2 \\
            & \leq 2 C e^{C\left(T_{N-1}-t_n\right)}\left(\left|x_n\right|^2+\left|y_n\right|^2+C\right)\left(T_{N-1}-t_n\right)+2\left(x_n-x\right)^2+2\left(y_n-y\right)^2,
        \end{split}
        \end{equation*}
        for some positive constant $C>0$. In the previous inequalities we used standard SDE estimates (see for example: (\cite{pham2009continuous}, Theorem 1.3.16)) and uniform boundedness of $Z$.  Combined with the continuity of $f_{N-1}$, we obtain:
        \begin{equation}\label{eq::limfn-1}
             \hat{V}(T^{-}_{N-1},x,y) = f_{N-1}(x,y).
        \end{equation}
        Applying the same argument to \eqref{v.region} taking a suitable sequence, we obtain: 
        \begin{align*}
            V^*(T_{N-1}^{-},x,y) \leq f_{N-1}(x,y),
        \end{align*}
        where for all $(t,x,y)\in [0,T]\times \RR\times [0,\infty)$, we define
        \begin{align*}
            V^*(t^{-},x,y) &= \limsup_{{(t',x',y') \rightarrow (t,x,y)}, \hspace{1mm} t'<t,\hspace{1mm}  y'>0}V(t',x',y'), \\ 
            V_*(t^{-},x,y) &= \liminf_{{(t',x',y') \rightarrow (t,x,y)}, \hspace{1mm} t'<t,\hspace{1mm}  y'>0 }V(t',x',y').
        \end{align*}
        Furthermore, we consider a sequence $(\tilde{t}_n,\tilde{x}_n,\tilde{y}_n) \mapsto (T_{N-1},x,y)$, such that $\lim_{n\rightarrow \infty} V(\tilde{t}_n,\tilde{x}_n,\tilde{y}_n) = V_{*}(T_{N-1}^{-},x,y)$. Then, 
        \begin{align*}
           \hat{V}(T_{N-1}^{-},x,y) = \lim_{n\rightarrow\infty} \hat{V}(\tilde{t}_n,\tilde{x}_n,\tilde{y}_n) \leq V_*(T_{N-1}^{-},x,y) \leq V^*(T_{N-1}^{-},x,y) \leq f_{N-1}(x,y).
        \end{align*}
        Combining the later inequality with \eqref{eq::limfn-1}, we obtain $V(T_{N-1}^{-},x,y) = f_{N-1}(x,y)$.\qed
        \\

Finally, we show the induction step. Assume that the results from Theorem $3$ hold on $\mathcal{R}_{i+1}$, for some $i \in \left\{1,\ldots, N-1\right\}$. We will show that it holds on $\mathcal{R}_i$. As the proofs for $2.2$, and $2.3$ are identical compared to the proof in the terminal region $\mathcal{R}_N$, we only write the induction step for Theorem $2.1$. 
\\

\textit{Proof of Theorem $2.1$ on $\mathcal{R}_i$}
\\
\\
Let $(t,x,y) \in \mathcal{R}_{i}$. Our hypothesis of induction claims that Theorem $2$ holds on $\mathcal{R}_{i+1}$. Hence, $V(T^{-}_{i},x,y) =x+ f_i(y) = x+\sup_{\eta} \left\{v(T_i,y-\eta)-\eta^{\gamma} : \hspace{1mm} \eta \in [0,y] \right\}$, for all $(x,y) \in \RR \times [0,\infty)$. By induction, we obtain that for $0 \leq\eta \leq y$: 
\begin{align*}
    v(T_i,y-\eta)+x+\eta^\gamma &\leq \phi^{\gamma,\delta_0}(T_i,y-\eta)+x+\eta^\gamma = \phi^{\gamma,\frac{1}{(N-i+1)^{\gamma-1}},\delta_0}(T_i,y-\eta)+x+\eta^\gamma.
\end{align*}
Hence, for all $(x,y) \in \RR \times (0,\infty)$, using that $\phi^{\gamma,\delta_0}(T_i,y) \geq v(T_i,y)$ by induction, we obtain 
\begin{align}
   V(T_i^{-},x,y) &=f_i(y)+x \nonumber \\
   &= \sup_{\eta}\left\{  v(T_i,y-\eta)+x-\eta^\gamma : \eta \in [0,y] \right\}\nonumber \\
   &\leq \sup_{\eta}\left\{ \phi^{\gamma,\delta_0}(T_i,y-\eta)-\eta^\gamma : \eta \in [0,y]\right\}+x \nonumber\\
   &\leq  \sup_{\eta} \left\{-\frac{1}{(N-i)^{\gamma-1}}(y-\eta)^{\gamma}-\eta^\gamma : \eta \geq 0\right\}+ be^{\frac{T-t}{T}}y^{\frac{1}{M}}+e^{c(T-t)}\left(1-e^{-y}\right)+x \nonumber \\
   &=-\frac{1}{(N-i+1)^{\gamma-1}}y^\gamma + be^{\frac{T-t}{T}}y^{\frac{1}{M}}+e^{c(T-t)}\left(1-e^{-y}\right)+x \nonumber \\
   &=\phi^{\gamma,\delta_0}(T_i^{-},y)+x,\label{eq.final}
\end{align}
where $\delta_0 := (b,c,M)$ is the constant found in Lemma \ref{lemma_supersol}.
Using Lemma \ref{lemma_supersol}  and \eqref{eq.final} we follow the same probabilistic argument we did for $i = N$, obtaining
\begin{equation*}
    -e^{k_a(T_i-t)}y+x\leq \hat{V}(t,x,y) \leq V(t,x,y) \leq \phi^{\gamma,\delta_0}(T_i^{-},y)+x,
\end{equation*}
for all $(t,x,y) \in (T_{i-1},T_i)\times \RR \times (0,\infty)$.

Hence, $\hat{V}_*$ is continuous on $(T_{i-1},T_{i})\times \{y= 0\}$. Invoking (Proposition $4.11$,\cite{bouchard2012weak} ) and following the same arguments done in the proof of this property in terminal region $\mathcal{R}_N$, we obtain that $V(t,x,y)=x+v(t,y)$, where $v$ is a continuous function on $\mathcal{R}_i$. Moreover, $v$ the unique (discontinuous) viscosity solution of the following state constrained HJB equation:
\begin{align*}
    -\varphi_t - \sup_{|z|\leq K}\left\{\frac{1}{2}\left(\varphi_y+\varphi_{yy}\right)z^2+z\right\} -   k_a y v_y   &= 0, \quad (t,y) \in [T_{i-1},T_i)\times (0,\infty),\\
    \varphi(T_{i},y) &=f_{i}(y). \nonumber
\end{align*}
in the class of functions with polynomial growth and  lower semi continuous envelope continuous at $[T_{i-1},T_i]\times\{y=0\}$. 
Hence, we show the induction step. \qed\\

The previous completes the proof of Theorem \ref{hjb:pp:in}. 

\bibliographystyle{alpha}
\bibliography{main}

\begin{thebibliography}{CWZ09}

\bibitem[BH07]{briandquadraticbsdeconvexgenerators2007}
Philippe Briand and Ying Hu.
\newblock Quadratic bsdes with convex generators and unbounded terminal conditions.
\newblock {\em Probability Theory and Related Fields}, 141:543–567, 2007.

\bibitem[BN12]{bouchard2012weak}
Bruno Bouchard and Marcel Nutz.
\newblock Weak dynamic programming for generalized state constraints.
\newblock {\em SIAM Journal on Control and Optimization}, 50(6):3344--3373, 2012.

\bibitem[BS91]{Convergenceviscositysougonathis}
Guy Barles and P.~E. Souganidis.
\newblock Convergence of approximation schemes for fully nonlinear second order equations.
\newblock {\em Asymptotic Analysis}, 4:271--283, 1991.

\bibitem[BS96]{stochcontroldiscrete}
Dimitry~P. Berstekas and Steven~E. Shreeve.
\newblock {\em Stochastic Control. A Discrete-Time Case.}
\newblock Athena Scientific, 1996.

\bibitem[CB06]{guide2006infinite}
D.~Aliprantis Charalambos and Kim~C. Border.
\newblock {\em Infinite dimensional analysis, a Hitchhiker's Guide}.
\newblock Springer, 2006.

\bibitem[CIL92]{usersguideviscosity}
Michael~G. Crandall, Hitoshi Iishi, and Piere-Louis Lions.
\newblock User's guide to viscosity solutions of second order partial differential equations.
\newblock {\em Bulletin of the American mathematical society}, 1992.

\bibitem[CPT18]{cvitanic2018dynamic}
Jak{\v{s}}a Cvitani{\'c}, Dylan Possama{\"\i}, and Nizar Touzi.
\newblock Dynamic programming approach to principal--agent problems.
\newblock {\em Finance and Stochastics}, 22:1--37, 2018.

\bibitem[CWZ06]{cvitanic2006optimal}
Jak{\v{s}}a Cvitani{\'c}, Xuhu Wan, and Jianfeng Zhang.
\newblock Optimal contracts in continuous-time models.
\newblock {\em International Journal of Stochastic Analysis}, 2006(1):095203, 2006.

\bibitem[CWZ08]{cvitanic2008principal}
Jaksa Cvitanic, Xuhu Wan, and Jianfeng Zhang.
\newblock Principal-agent problems with exit options.
\newblock {\em The BE Journal of Theoretical Economics}, 8(1), 2008.

\bibitem[CWZ09]{cvitanic2009optimal}
Jak{\v{s}}a Cvitani{\'c}, Xuhu Wan, and Jianfeng Zhang.
\newblock Optimal compensation with hidden action and lump-sum payment in a continuous-time model.
\newblock {\em Applied Mathematics and Optimization}, 59:99--146, 2009.

\bibitem[CZ24]{capponizhang}
Agostino Capponi and Yuchong Zhang.
\newblock A continuous time framework for sequential goal-based wealth management.
\newblock {\em Management Science}, 2024.

\bibitem[{\'E}MP19]{manymanyagents}
Romuald {\'E}lie, Thibaut Mastrolia, and Dylan Possamaï.
\newblock A tale of a principal and many, many agents.
\newblock {\em Mathematics of Operations Research}, 44(2), 2019.

\bibitem[{\'E}P19]{competitive}
Romuald {\'E}lie and Dylan Possama{\"\i}.
\newblock Contracting theory with multiple competitive agents.
\newblock {\em SIAM Journal of Control and Optimization}, 57(2), 2019.

\bibitem[ET13]{ElKarouiTan2024}
Nicole El{\ }Karoui and Xialu Tan.
\newblock Capacities, measurable selection and dynamic programming part {II}: Application in stochastic control problems.
\newblock {{}}, {É}cole Polytechnique and Universit{é} Paris-Dauphin{é}, 2013.

\bibitem[FHM90]{fudenberg1990short}
Drew Fudenberg, Bengt Holmstrom, and Paul Milgrom.
\newblock Short-term contracts and long-term agency relationships.
\newblock {\em Journal of economic theory}, 51(1):1--31, 1990.

\bibitem[FT90]{fudenberg1990moral}
Drew Fudenberg and Jean Tirole.
\newblock Moral hazard and renegotiation in agency contracts.
\newblock {\em Econometrica: Journal of the Econometric Society}, 58(6):1279--1319, 1990.

\bibitem[Hel07]{hellwig2007role}
Martin~F Hellwig.
\newblock The role of boundary solutions in principal--agent problems of the {H}olmstr{\"o}m--{M}ilgrom type.
\newblock {\em Journal of Economic Theory}, 136(1):446--475, 2007.

\bibitem[HM87]{holmstrom1987aggregation}
Bengt Holmstrom and Paul Milgrom.
\newblock Aggregation and linearity in the provision of intertemporal incentives.
\newblock {\em Econometrica: Journal of the Econometric Society}, 55(2):303--328, 1987.

\bibitem[HP24]{hernandezpossmai}
Camilo Hern{\'a}ndez and Dylan Possama{\"\i}.
\newblock Time-inconsistent contract theory.
\newblock {\em Mathematical Finance}, 34(3):737–1085, 2024.

\bibitem[HS02]{hellwig2002discrete}
Martin~F Hellwig and Klaus~M Schmidt.
\newblock Discrete--time approximations of the {H}olmstr{\"o}m--{M}ilgrom brownian--motion model of intertemporal incentive provision.
\newblock {\em Econometrica}, 70(6):2225--2264, 2002.

\bibitem[HSM19]{hernandez2019contract}
Nicolas Hern{\'a}ndez-Santibanez and Thibaut Mastrolia.
\newblock Contract theory in a vuca world.
\newblock {\em SIAM Journal on Control and Optimization}, 57(4):3072--3100, 2019.

\bibitem[Lam83]{lambert1983long}
Richard~A Lambert.
\newblock Long-term contracts and moral hazard.
\newblock {\em The Bell Journal of Economics}, 14(2):441--452, 1983.

\bibitem[M{\"u}l97]{muller1997first}
Holger~M M{\"u}ller.
\newblock The first-best sharing rule in the continuous-time principal-agent problem with exponential utility.
\newblock {{}}, Stockholm School of Economics, 1997.

\bibitem[M{\"u}l00]{muller2000asymptotic}
Holger~M M{\"u}ller.
\newblock Asymptotic efficiency in dynamic principal-agent problems.
\newblock {\em Journal of Economic Theory}, 91(2):292--301, 2000.

\bibitem[Pha09]{pham2009continuous}
Huy{\^e}n Pham.
\newblock {\em Continuous-time stochastic control and optimization with financial applications}, volume~61.
\newblock Springer Science \& Business Media, 2009.

\bibitem[PT24]{possamai2024there}
Dylan Possama{\"\i} and Nizar Touzi.
\newblock Is there a golden parachute in {S}annikov’s principal--agent problem?
\newblock {\em Mathematics of Operations Research}, 0(0), 2024.

\bibitem[Rog85]{rogerson1985repeated}
William~P Rogerson.
\newblock Repeated moral hazard.
\newblock {\em Econometrica: Journal of the Econometric Society}, 53(1):69--76, 1985.

\bibitem[RY83]{rubinstein1983repeated}
Ariel Rubinstein and Menahem~E Yaari.
\newblock Repeated insurance contracts and moral hazard.
\newblock {\em Journal of Economic theory}, 30(1):74--97, 1983.

\bibitem[San08]{sannikov2008continuous}
Yuliy Sannikov.
\newblock A continuous-time version of the principal-agent problem.
\newblock {\em The Review of Economic Studies}, 75(3):957--984, 2008.

\bibitem[SS93]{schattler1993first}
Heinz Sch{\"a}ttler and Jaeyoung Sung.
\newblock The first-order approach to the continuous-time principal--agent problem with exponential utility.
\newblock {\em Journal of Economic Theory}, 61(2):331--371, 1993.

\bibitem[Sun95]{sung1995linearity}
Jaeyoung Sung.
\newblock Linearity with project selection and controllable diffusion rate in continuous-time principal-agent problems.
\newblock {\em The RAND Journal of Economics}, pages 720--743, 1995.

\bibitem[Sun97]{sung1997corporate}
Jaeyoung Sung.
\newblock Corporate insurance and managerial incentives.
\newblock {\em Journal of Economic Theory}, 74(2):297--332, 1997.

\bibitem[Wil11]{williams2011persistent}
Noah Williams.
\newblock Persistent private information.
\newblock {\em Econometrica}, 79(4):1233--1275, 2011.

\bibitem[Wil15]{williams2015solvable}
Noah Williams.
\newblock A solvable continuous time dynamic principal--agent model.
\newblock {\em Journal of Economic Theory}, 159:989--1015, 2015.

\end{thebibliography}

\end{document}